\documentclass{imanum}
\usepackage{enumerate}
\usepackage{amsmath,amssymb,units,mathtools,arydshln,multirow,mathdots}
\usepackage{algorithm,algorithmicx,algpseudocode,listings}
\usepackage{graphicx,tikz,pgfplots,caption,placeins,flafter}
\usepackage{subcaption}
\newlength\figurewidth 
\jno{}
\received{}

\usepackage{color}

\definecolor{green_}{RGB}{0,180,0}%
\definecolor{green_2}{RGB}{0,80,0}%
\definecolor{blue_}{RGB}{0,0,180}%
\definecolor{blue_light}{RGB}{0,180,180}%

\definecolor{mygreen}{RGB}{28,172,0} % color values Red, Green, Blue
\definecolor{mylilas}{RGB}{170,55,241}

\newcommand{\R}{\mathbb{R}}
\newcommand{\C}{\mathbb{C}}
\newcommand{\comm}[1]{\textcolor{red}{}}
\DeclareMathOperator{\rank}{rank}
\DeclareMathOperator{\Ker}{Ker}
\DeclareMathOperator{\diag}{\mathrm{diag}}
\newcommand{\bsym}[1]{\boldsymbol{#1}}

\title{Automatic rational approximation and linearization of nonlinear eigenvalue problems}
\shorttitle{Automatic rational approximation for nonlinear eigenvalue problems}
\author{%
{\sc
	Pieter Lietaert \thanks{Department of Computer Science, KU Leuven, University of Leuven, 3001 Heverlee, Belgium. 
		Supported by KU Leuven Research Council Grant OT/14/074 and BELSPO Grant IAP~VII/19.},
		Javier P{\'e}rez \thanks{
Department of Mathematical Sciences, University of Montana, MT, USA. The
research of J. P´erez was partially supported by KU Leuven Research Council grant OT/14/074 and the Interuniversity Attraction
Pole DYSCO, initiated by the Belgian State Science Policy Office.
		},
		Bart Vandereycken \thanks{
Section of Mathematics, University of Geneva, Rue du Li\`evre 2-4, 1211 Geneva, Switzerland 
		}
{\sc and }\\
Karl Meerbergen \footnotemark[2]}
}
\begin{document}
\maketitle

\begin{abstract}
{We present a method for solving nonlinear eigenvalue problems using rational approximation. The method uses the AAA method by Nakatsukasa, S\`{e}te, and Trefethen to approximate the nonlinear eigenvalue problem by a rational eigenvalue problem and is embedded in the state space representation of a rational polynomial by Su and Bai.
The advantage of the method, compared to related techniques such as NLEIGS and infinite Arnoldi, is
the efficient computation by an automatic procedure.
In addition, a set-valued approach is developed that allows building a low degree rational approximation of a nonlinear eigenvalue problem.
The method perfectly fits the framework of the Compact rational Krylov methods (CORK and TS-CORK), 
allowing to efficiently solve large scale nonlinear eigenvalue problems.
Numerical examples show that the presented framework is competitive with NLEIGS and usually produces smaller linearizations with the same accuracy but with less effort for the user.
}{Nonlinear eigenvalue problem, Rational interpolation, Rational Krylov method}
\end{abstract}

\section{Introduction}

The nonlinear eigenvalue problem (NEP) is the problem of finding  scalars $\lambda\in\mathbb{C}$ and nonzero vectors $x,y\in\mathbb{C}^n$ such that
\begin{equation}\label{eq:eigval}
A(\lambda)x =0 \quad \mbox{and} \quad y^*A(\lambda)=0,
\end{equation}
where $A\colon\mathbb{C}\to\mathbb{C}^{n\times n}$ is a nonlinear matrix valued function.
The scalar $\lambda\in\mathbb{C}$ is called an eigenvalue and the vectors $x$ and $y$ are called, respectively, associated right and left eigenvectors.
Usually, one is interested in the eigenvalues in a specific region $\Sigma\subset\mathbb{C}$.
We assume that $A$ in \eqref{eq:eigval} is regular, i.e., there is at least one $\lambda\in\mathbb{C}$ for which $\det(A(\lambda))\neq0$.

Numerical methods for computing the eigenvalues of generic nonlinear eigenvalue problems in a region $\Sigma$ are based on approximation theory.
There currently are two classes of methods: those based on contour integrals \citep[][]{beyn12}, and those based on rational and polynomial approximation, e.g. \cite{efkr12}, infinite Arnoldi \citep[][]{jamm12} and NLEIGS \citep[][]{gvmm14}.
The methods based on contour integration rely on Keldysh theorem, where the eigenvalues in $\Sigma$ are found as the poles of a resolvent using contour integration.
There are several variations of this approach, see \cite{beyn12}.
% \cite{sakurai}.
For large scale problems, the Jacobi--Davidson method is combined with contour integration in \cite{effe13} for computing invariant pairs of \eqref{eq:eigval}.
The second class of methods approximates $A$ by a polynomial or rational function on $\Sigma$ and solves the resulting polynomial or rational eigenvalue problem.
Polynomial and rational eigenvalue problems can be solved by Krylov methods through a linearization. The prototype linearization is the companion pencil.

In this paper, we focus on methods that build a rational approximation, reformulate the resulting problem as a linear eigenvalue problem (by the process of linearization) and then use a Krylov method for solving the linearized problem.
Roughly speaking, there exist three approaches for rational approximation. The NLEIGS method uses potential theory for the selection of poles and interpolation points, and embeds this within a rational polynomial expressed in rational Newton basis \citep[see][]{gvmm14}.
%The disadvantage of this method is that the nonlinear functions need to be analyzed beforehand in order to apply potential theory. 
%Even then, an unambiguous choice of approximant is not always clear.
The second is the infinite Arnoldi method \citep[][]{jamm12} that uses the discretization of an infinite dimensional operator that is a linear representation of the nonlinear eigenvalue problem.
The discretization of this operator leads to a finite dimensional linear problem that is solved by the Arnoldi method.
%The disadvantages of this approach are that a generalization to rational Krylov methods is, to date, not known, and that the spectral discretization does not always converge fast for generic functions. However, it works very well for delay differential equations.
The third approach expresses a Pad\'e approximation in state-space form and applies a Krylov method to a linearization, see \cite{suba11}.

%\begin{bv} Why do we comment on advantage / disadvantage for infinite Arnoldi, and not the other two ? \end{bv}
%\begin{pl} 
%	Added disadvantage for the first method. If I understand correctly, the last method is very similar to our method, but starts from a rational in stead of general nonlinear problem.
%\end{pl}

The approach of this paper is inspired by \cite{suba11} and assumes that the matrix valued function $A$ can be written as
\begin{equation}
\label{eq:NEP}
A(\lambda)= P(\lambda) + G(\lambda),
\end{equation}
where $P(\lambda)$ is a matrix polynomial and $G$ is an arbitrary matrix valued function of the form
\begin{equation}\label{eq:G}
G(\lambda) = \sum_{i=1}^s (C_i-\lambda D_i)g_i(\lambda),
\end{equation}
where $C_i,D_i$ are constant $n\times n$ matrices and $g_i \colon \mathbb{C}\rightarrow\mathbb{C}$ is a nonlinear function.
Using the approach of \cite{suba11}, each $g_i$ can be approximated by a different rational function with different poles and interpolation points, determined independently from each other.
This in contrast to NLEIGS \citep[][]{gvmm14}, where the same poles and nodes are used for all terms.
For efficiency, it is assumed that the number of nonlinear terms, $s$, is modest.

The contribution of this paper is threefold.
First, the rational approximations used in this work are obtained by employing the adaptive Antoulas--Anderson (AAA) algorithm introduced by Nakatsukasa, S\`{e}te, and Trefethen in \cite{nast16}.
This approach presents two key advantages:
the AAA algorithm is not domain-dependent, i.e., it works effectively even with  sets that may include  disconnected regions of irregular shape, possibly unbounded; and
once the approximation region has been fixed, it is the algorithm and not the user who chooses the number of poles and zeros of the rational approximation and their values in an adaptive way.
Hence, unlike other methods, neither special knowledge of the nonlinear functions nor advance knowledge of complex analysis is required from the user.

The second contribution is
an automatic strategy that uses the same poles and interpolation points for (a subset of) all $g_i$, which leads to the same appealing properties as NLEIGS, when $s$ is not so small.
%Note, however, that exploiting low rank in $A_j$ has been shown to be beneficial \cite{gvmm14} \cite{jamm12}(?).
%We therefore do not claim that the approaches in \cite{suba11} \cite{dopi17} are not competitive with our, more general, approach.
Numerical experiments compare with rational approximations obtained using NLEIGS in \cite{gvmm14}. % and infinite Arnoldi \cite{jamm12}.

The third contribution of the paper lies in the observation that
the linearization from \cite{suba11} fits perfectly in the framework of the CORK (Compact Rational Krylov) method in \cite{bemm15} and the two-sided CORK method in \cite{limt17},
which makes the linearization suitable for large scale problems and rational functions of high degree.
There is no need to require $C_i,D_i$ to be of low rank but if they are it can be exploited. %as in \cite{suba11} and \cite{dopi17}.

The rest of the paper is structured as follows.
Section~\ref{sec:AAA} presents the original AAA approximation of a nonlinear function $g$ and its generalization to a set of nonlinear functions $g_1,\ldots,g_s$.
In \S\ref{sec:lin}, we reformulate the linearization by \cite{suba11} as a CORK linearization and show a relation between the eigenvalues and eigenvector of the linearization and the rational matrix polynomial,
including the case of low rank $C_i-\lambda D_i$, $i=1,\ldots,s$.
Section~\ref{sec:examples} illustrates the AAA approaches by solving nonlinear eigenvalue problems and compares with NLEIGS.
Section~\ref{sec:conlusions} is reserved for the conclusions.

\section{Scalar rational approximations by AAA}\label{sec:AAA}

As explained in the introduction, we intend to use rational approximations for the NEP. As a first step, this requires approximating the scalar functions $g_i(\lambda)$ in~\eqref{eq:G}, either separately or together, by rational functions. Our approach is based on a recently introduced  algorithm from~\cite{nast16} that we review next.

\subsection{The AAA algorithm}

Let $g\colon \C \to \C$ denote a generic nonlinear function that we would like to approximate on $\Sigma \subset \C$ by a rational function $r(\lambda)$. The adaptive Antoulas--Anderson (AAA) algorithm from~\cite{nast16}, will construct this function $r(\lambda)$ in barycentric form: 
\begin{equation}\label{eq:bary_approx}
 r(\lambda)=\underbrace{\sum_{j=1}^{m} \frac{g(z_j)\, \omega_j}{\lambda-z_j}}_{=:n_m(\lambda)} \Big{/} \underbrace{\sum_{j=1}^{m} \frac{\omega_j}{\lambda-z_j}}_{=:d_m(\lambda)}.
 \end{equation}
Here, $z_1,\hdots,z_m$ are a set of distinct \emph{support points} and $\omega_1,\hdots,\omega_m$ are the \emph{weights}.
 Note that, as long as $\omega_j\neq 0$,  $\lim_{\lambda\rightarrow z_j}r(\lambda)=g(z_j)$.  
In other words, the rational function~\eqref{eq:bary_approx} interpolates the function  $g(\lambda)$ at $z_1,\ldots,z_m$.

The AAA algorithm computes the support points and the weights iteratively by minimizing the linearized residual of the rational approximation on a sample set $Z$ of $M$ points. 
The set $Z$ can be seen as a sufficiently fine discretization of the region $\Sigma$ which means $M$ is typically quite large, say, $10^4$. At the $m$th step of the algorithm, the next support point $z_m$ is chosen where the residual $g(\lambda)-n_{m-1}(\lambda)/d_{m-1}(\lambda)$ attains its maximum absolute value. Then, denoting
\begin{align*}
&Z^{(m)}=\{Z^{(m)}_1,\ldots,Z^{(m)}_{M-m}\}:=Z/\{z_1,\ldots,z_m\} \quad \mbox{and} \\
&G^{(m)}=\{G^{(m)}_1,\ldots,G^{(m)}_{M-m}\}:=g(Z^{(m)}),
\end{align*}
it computes the vector of weights $\omega=\begin{bmatrix} \omega_1 & \cdots & \omega_m \end{bmatrix}^T$ with $\|\omega\|_2=1$ that minimizes the 2-norm of the linearized residual
\begin{equation}
\label{eq:AAA-residual}
\begin{bmatrix}
g(Z_1^{(m)})d_m(Z_1^{(m)})-n_m(Z_1^{(m)}) \\ \vdots \\
g(Z_{M-m}^{(m)})d_m(Z_{M-m}^{(m)})-n_m(Z_{M-m}^{(m)})
\end{bmatrix} =
\begin{bmatrix}
\frac{G_1^{(m)}-g(z_1)}{Z_1^{(m)}-z_1} &\cdots & \frac{G_1^{(m)}-g(z_m)}{Z_1^{(m)}-z_m}  \\
\vdots & \ddots & \vdots \\
\frac{G_{M-m}^{(m)}-g(z_1)}{Z_{M-m}^{(m)}-z_1} &\cdots& \frac{G_{M-m}^{(m)}-g(z_m)}{Z_{M-m}^{(m)}-z_m}
\end{bmatrix}
\begin{bmatrix}
\omega_1\\ \vdots \\ \omega_m
\end{bmatrix}.
\end{equation}
This can be done by using the SVD on the Loewner matrix above. The procedure terminates when the norm of the residual \eqref{eq:AAA-residual} is less than a user defined tolerance, for example, $10^{-13}$.
For further details of the AAA algorithm, including the removal of numerical Froissart doublets, we refer to~\cite{nast16} and our (modified) implementation in App.~\ref{app:setAAA}.

A key feature of AAA is its flexibility in selecting the domain of approximation (through the set $Z$), unlike other methods that are domain-dependent.  Furthermore, the user only needs to supply this domain and a tolerance. Then, the poles and zeros of the rational interpolant~\eqref{eq:bary_approx} are found automatically by the algorithm. In many cases, AAA succeeds in computing a rational interpolant that is close to the optimal one in min-max sense. However, the algorithm can fail on difficult functions; see~\cite{Filip:2017} for examples. In the numerical experiments, however, we did not see such pathological behavior and AAA performed adequately.

In Section~\ref{sec:lin}, it will be convenient to write the rational functions from AAA that are in barycentric form into an equivalent state-space form.
\begin{proposition}\label{prop:state-space}
The rational function  \eqref{eq:bary_approx} can be written as
\begin{equation}\label{eq:f-approx}
r(\lambda) = 
\begin{bmatrix}
g(z_1) \omega_1 & \cdots & g(z_{m})\omega_{m}
\end{bmatrix}
\begin{bmatrix}
\omega_1 & \omega_2 & \cdots & \omega_{m-1} & \omega_{m} \\
\lambda - z_1 & z_2-\lambda \\
& \lambda - z_2 & \ddots \\
& & \ddots & z_{m-1}-\lambda \\
& & & \lambda-z_{m-1} & z_{m} - \lambda
\end{bmatrix}^{-1}
\begin{bmatrix}
1 \\0 \\ \vdots \\ 0
\end{bmatrix},
\end{equation}
where the entries that are not depicted are equal to zero.
\end{proposition} 
\begin{proof}
Let $d(\lambda)=\sum_{j=1}^{m}\omega_j(\lambda-z_j)^{-1}$ and $n(\lambda)=\sum_{j=1}^{m}g(z_j)\omega_j(\lambda-z_j)^{-1}$ denote, respectively, the denominator and numerator of $r(\lambda)$ in \eqref{eq:rat_approx}.
Then, it is easily verified that the vector
\[
\frac{1}{d(\lambda)}
\begin{bmatrix}
(\lambda-z_1)^{-1} \\ \vdots \\ (\lambda-z_{m})^{-1} 
\end{bmatrix}
\]
is the first column of  
\[
\begin{bmatrix}
\omega_1 & \omega_2 & \cdots & \omega_{m-1} & \omega_{m} \\
\lambda - z_1 & z_2-\lambda \\
& \lambda - z_2 & \ddots \\
& & \ddots & z_{m-1}-\lambda \\
& & & \lambda-z_{m-1} & z_{m} - \lambda
\end{bmatrix}^{-1}.
\]
Thus, we obtain the desired result as
\begin{displaymath}
\frac{1}{d(\lambda)}\begin{bmatrix}
g(z_1)\omega_1 & \cdots & g(z_{m})\omega_{m}
\end{bmatrix}
\begin{bmatrix}
(\lambda-z_1)^{-1} \\ \vdots \\ (\lambda-z_{m})^{-1} 
\end{bmatrix}=
\frac{n(\lambda)}{d(\lambda)}=r(\lambda). 
\end{displaymath}
\end{proof}

\subsection{A set-valued AAA algorithm}\label{sec:set-valued AAA}

There are applications where a large number of nonlinear functions $g_1,\hdots,g_s$ need to be approximated over the same region of the complex plane.
One can of course use the AAA algorithm on each function separately but, as we will see in the numerical examples in \S\ref{sec:examples},  it is sometimes beneficial to find support points and poles that work for all the functions at the same time. The result is then a linearization with a  smaller total degree compared to the linearization obtained from the separate applications of AAA. 
% However, this typically leads to a large linearization with additional inevitable costs which makes this approach less appealing for $s>1$.
% The numerical examples in \S\ref{sec:examples} show that finding support points and poles that work for all the functions at the same time might lead to a smaller total degree.
In this section, we show how to extend the AAA approach to accomplish this.

Let $g_1,\hdots,g_s$ be nonlinear functions of the same scale, which can be accomplished by simply scaling them as $g_i(\lambda) / \max_j|g_i(z_j)|.$
Our aim is to construct rational approximations to $g_1,\hdots,g_s$ of the form
\begin{equation}\label{eq:bary_approx-set-valued}
g_i(\lambda)\approx \underbrace{\sum_{j=1}^{m} \frac{g_i(z_j)\, \omega_j}{\lambda-z_j}}_{=:n_{i,m}(\lambda)} \Big{/} \underbrace{\sum_{j=1}^{m} \frac{\omega_j}{\lambda-z_j}}_{=:d_m(\lambda)}.
 \end{equation} 
Note that all rational approximants share the same support points $z_j$ and weights $\omega_j$. 
In the spirit of the AAA algorithm, these support points and weights are computed iteratively.
At the $m$th step, the next support point $z_m$ is chosen where the maximum of the residuals, i.e.,
\begin{equation*}
	\max_{i}{\left|g_i(\lambda)-\frac{n_{i,m-1}(\lambda)}{d_{m-1}(\lambda)}\right|},
\end{equation*}
attains its maximum. Then, denoting
\begin{align*}
&Z^{(m)}=\{Z^{(m)}_1,\ldots,Z^{(m)}_{M-m}\}:=Z/\{z_1,\ldots,z_m\} \quad \mbox{and} \\
&G_i^{(m)}=\{G^{(m)}_{i,1},\ldots,G^{(m)}_{i,M-m}\}:=g_i(Z^{(m)}),
\end{align*}
the residual vector \eqref{eq:AAA-residual} can be written to incorporate different functions:
\begin{equation}
\label{eq:AAA-residual-multf}
	\begin{bmatrix}
		\frac{G_{1,1}^{(m)}-g_1(z_1)}{Z_1^{(m)}-z_1} &\cdots & \frac{G_{1,1}^{(m)}-g_1(z_m)}{Z_1^{(m)}-z_m}  \\
		\vdots & \ddots & \vdots \\
		\frac{G_{1,M-m}^{(m)}-g_1(z_1)}{Z_{M-m}^{(m)}-z_1} &\cdots& \frac{G_{1,M-m}^{(m)}-g_1(z_m)}{Z_{M-m}^{(m)}-z_m} \\
		\frac{G_{2,1}^{(m)}-g_2(z_1)}{Z_1^{(m)}-z_1} &\cdots & \frac{G_{2,1}^{(m)}-g_2(z_m)}{Z_1^{(m)}-z_m}  \\
		\vdots & \ddots & \vdots \\
		\frac{G_{2,M-m}^{(m)}-g_2(z_1)}{Z_{M-m}^{(m)}-z_1} &\cdots& \frac{G_{2,M-m}^{(m)}-g_2(z_m)}{Z_{M-m}^{(m)}-z_m} \\
		\vdots & \vdots & \vdots \\
		\frac{G_{s,1}^{(m)}-g_s(z_1)}{Z_1^{(m)}-z_1} &\cdots & \frac{G_{s,1}^{(m)}-g_s(z_m)}{Z_1^{(m)}-z_m}  \\
		\vdots & \ddots & \vdots \\
		\frac{G_{s,M-m}^{(m)}-g_s(z_1)}{Z_{M-m}^{(m)}-z_1} &\cdots& \frac{G_{s,M-m}^{(m)}-g_s(z_m)}{Z_{M-m}^{(m)}-z_m} \\
	\end{bmatrix}
	\begin{bmatrix}
		\omega_1\\ \vdots \\ \omega_m
	\end{bmatrix},
\end{equation}
The vector of weights $\omega=\begin{bmatrix} \omega_1 & \cdots & \omega_m \end{bmatrix}^T$ is computed as the vector minimizing the norm of \eqref{eq:AAA-residual-multf} under the constraint $\|\omega\|_2=1$.
We note that minimizing this norm is equivalent to minimizing the sum of squares of the norms of the residual vectors of the different functions.

The computational cost of the AAA algorithm and of our modified set-valued version  might become an issue when a large number of nonlinear functions need to be approximated. %as one has to deal with $s$ matrices of size $M$ or with a matrix of size $sM$.
To partly alleviate this extra cost, we can slightly reformulate how the AAA algorithm solves the least squares problems in \eqref{eq:AAA-residual} and \eqref{eq:AAA-residual-multf} to compute the vector of weights. 
This is outlined in the next section.

\subsection{Efficient solution of the least squares problem in AAA}
The original AAA algorithm requires the repeated singular value decomposition of the tall but skinny Loewner matrix in~\eqref{eq:AAA-residual} to find its right singular vector  corresponding to the smallest singular value.
This can become costly for large matrices, which is certainly the case if we have do this for a large number of functions as in~\eqref{eq:AAA-residual-multf}. Fortunately, we can exploit that the Loewner matrices in~\eqref{eq:AAA-residual} or in~\eqref{eq:AAA-residual-multf} differ only in a few rows and columns throughout each iteration of the AAA algorithm. In particular,  by computing economy-size QR decompositions of said matrices, we only need to obtain the right singular vectors of much smaller matrices. In turn, these QR decompositions can be computed using updating strategies.
% ; see, e.g., \cite{dgks76, Hammarling:2008}.

Let $L_m$ be the Loewner matrix in~\eqref{eq:AAA-residual} or~\eqref{eq:AAA-residual-multf}, of size $n\times m$.
We recall that $n\gg m$.
The matrix $L_m$ can be stored as
\[
L_m = QH,
\]
where $Q$ is an $n\times m$ matrix with orthonormal columns, and $H$ is an $m\times m$ matrix. 
Note that the right singular vectors of the matrix $L_m$ can be computed as the right singular vectors of the small matrix $H$.
The column of the matrix $Q$ can be iteratively found by adding one column of the Loewner matrix $L_m$ in each step and applying Gram-Schmidt orthogonalization.
Note, however, that in each step of AAA, the number of rows of $Q$ is reduced by one for \eqref{eq:AAA-residual} and by $s$ for \eqref{eq:AAA-residual-multf} because the support points $z_i$ are removed from the set $Z$ that defines the residual.
With the removal of these rows, $Q$ is no longer orthogonal.
However, we can reorthogonolize $Q$ cheaply as follows.
Let $Q_{\rm r} \in \mathbb{C}^{r \times m}$ be the matrix whose rows are the rows that have been removed from $Q$ in step $m$ and let $\widetilde{Q}$ be the matrix obtained from $Q$ after the removal of these rows.
We then have, since $Q$ is orthogonal,
\begin{equation*}
	\widetilde{Q}^*\widetilde{Q} = I_{m} - Q_{\rm r}^*Q_{\rm r}.
\end{equation*}
So, by taking the Cholesky decomposition,
\begin{equation*}
	I_{m} - Q_{\rm r}^*Q_{\rm r} = S^*S,
\end{equation*}
we have that matrix $\widetilde{Q}S^{-1}$ is orthogonal, and we update $H = SH$.
We can further avoid the (costly) explicit multiplication $\widetilde{Q}S^{-1}$ by storing $S$ in matrix
\begin{equation*}
	S_{m} = \left[
		\begin{array}{cc}
			S_{m-1} \\
			& 1
		\end{array}
	\right]S^{-1},
\end{equation*}
which is only used in matrix-vector and matrix-matrix multiplications with vector and matrices of size $O(m)$, that is, of small size.

This procedure is implemented in our MATLAB version of the AAA algorithm in App.~\ref{app:setAAA}, which also shows how to rework the AAA algorithm to incorporate multiple functions, that is, the set-valued AAA algorithm.
It should be compared with the algorithm presented in \cite{nast16}.
The main cost of the algorithm is reduced to the Gram-Schmidt orthogonalization process of the long vectors of the Loewner matrix.

%\begin{bv} Add reference to updating QR factorizations \end{bv}
	
\section{Rational approximations for NEPs using AAA}\label{sec:lin}

In this section we show how the scalar rational functions, computed by AAA, can be used efficiently to obtain a rational approximation of the NEP. In particular, we will present linearizations that build on the CORK~\citep[][]{bemm15} and the TS-CORK~\citep[][]{limt17} frameworks and exploit possible low-rank terms. These frameworks allow that the eigenvalues of the linearizations can be computed efficiently by, for example, the rational Krylov method.

\subsection{The CORK framework}\label{sec:CORK-framework}
The starting point of the compact rational Krylov (CORK) method in~\cite{bemm15} is a matrix-valued function of the form
\begin{equation}\label{eq:P}
P(\lambda)=\sum_{i=0}^{k-1} (A_i-\lambda B_i)f_i(\lambda), \quad \mbox{with} \quad A_i,B_i\in\mathbb{C}^{n\times n},
\end{equation}
where $f_i\colon\mathbb{C}\rightarrow \mathbb{C}$ are  polynomial or rational functions satisfying the linear relation 
\begin{equation}\label{eq:linear relation}
(M-\lambda N)\, f(\lambda) =0 \qquad \text{with $\rank(M-\lambda N) = k-1$ for all $\lambda \in \C$}
\end{equation}
and $f(\lambda) = \begin{bmatrix}
f_0(\lambda) & \cdots  & f_{k-1}(\lambda)
\end{bmatrix}^T \neq 0$. 
Without much loss of generality, we further assume that $f_0(\lambda)\equiv 1$ has degree zero.
%and, if we write
%\begin{equation}\label{eq:partitioned pencil}
%M-\lambda N = \begin{bmatrix}
%m_0-\lambda n_0 & M_1-\lambda N_1
%\end{bmatrix} \quad \mbox{where} \quad M_1-\lambda N_1 \in\mathbb{C}[\lambda]^{k\times k},
%\end{equation}
%that $M_1-\lambda N_1$ is nonsingular for any $\lambda\in\mathbb{C}$.
This assumption is indeed not restrictive, as it covers most of the important cases in applications including monomials, Chebyshev polynomials, orthogonal polynomials, Newton polynomials, rational Newton functions. For a more general setting and explicit examples of $M - \lambda N$, we refer to~\cite{bemm15}.

%\begin{bv}
%	Move this assumption to section 1?
%\end{bv}	

Given a matrix-valued function~\eqref{eq:P} satisfying \eqref{eq:linear relation}, the matrix pencil
\begin{equation}\label{eq:lin-CORK}
\mathcal{L}_P(\lambda) = 
\left[
	\begin{matrix}
		\begin{array}{ccc}
	A_0-\lambda B_0 & \cdots & A_{k-1}-\lambda B_{k-1} \\
	\hline
	\end{array}\\
(M-\lambda N)\otimes I_n
\end{matrix}\right],
\end{equation}
is called the CORK linearization of $P(\lambda)$. When $P(\lambda)$ is a matrix polynomial, the pencil \eqref{eq:lin-CORK} is a linearization in the usual sense~\citep[][]{gohberg_lancaster}, since it corresponds to a ``block minimal bases pencil'', see \cite{dlpv17,rvv16}.  
When $P(\lambda)$ is a rational matrix, it is not clear whether the pencil~\eqref{eq:lin-CORK} is a linearization of $P(\lambda)$ in the sense of \cite{amparan_dopico}. In any case, for our purposes, we only need to use that $P(\lambda)$ and $\mathcal{L}_P(\lambda)$ have the same eigenvalues and that the eigenvectors of $P(\lambda)$ can be easily recovered from those of $\mathcal{L}_P(\lambda)$ \cite[see][Corollary 2.4]{bemm15}.

The CORK linearization~\eqref{eq:lin-CORK} is of size $kn \times kn$ which can become quite large. Fortunately, its Kronecker structure can be exploited when computing its eigenvalues by Krylov methods; see, e.g., \cite[Algorithm 3]{bemm15} on how to efficiently use the rational Krylov method in this context.

%However, for the nonlinear eigenvalue problem associated with \eqref{eq:NEP}, the presence of the nonlinear term \eqref{eq:G} prevents us from applying this approach. 
%In Section \ref{sec:CORK-AAA}, we will deal with this problem.
%The standard way of dealing with this problem is by applying some sort of rational approximation to \eqref{eq:G}, which is the subject of the following section.

%In this section, we show that one can construct CORK and generalized\footnote{Not defined what is meant by this} CORK linearizations of a rational approximation of \eqref{eq:NEP} by using the rational interpolants provided by the AAA algorithm.

\subsection{Extending CORK with AAA}

Let us now consider the NEP with $A(\lambda)$ as defined in~\eqref{eq:NEP}, that is,
\[
 A(\lambda)= P(\lambda) + \sum_{i=1}^s (C_i-\lambda D_i)g_i(\lambda).
\]
Using AAA, or its set-valued generalization, we can approximate each function $g_i(\lambda)$ on the region $\Sigma \subset \C$ as
\begin{equation}\label{eq:rat_approx}
g_i(\lambda)\approx r_i(\lambda) = \sum_{j=1}^{\ell_i} \frac{g_i(z_j^{(i)})\, \omega^{(i)}_j}{\lambda-z_j^{(i)}}\Big{/} \sum_{j=1}^{s\ell_i} \frac{\omega^{(i)}_j}{\lambda-z_j^{(i)}},
\end{equation}
where $\ell_i$ is the number of support points $z_j^{(i)}$ and weights $\omega^{(i)}_j$ for each $i=1,\ldots,s$. If some of the $g_i$ are approximated together by the set-valued AAA algorithm, the $z_j^{(i)}$ and $\omega^{(i)}_j$ are the same for their corresponding indices $i$. For now, we ignore this property. In any case, we can use the rational approximations $r_i(\lambda)$ to obtain an approximation of the NEP on the same region $\Sigma$:
\begin{equation}\label{eq:REP}
A(\lambda) \approx R(\lambda)=P(\lambda)+\sum_{i=1}^s (C_i-\lambda D_i)r_i(\lambda).
\end{equation}

We now show how to obtain a CORK-like linearization of $R(\lambda)$. If we assume that $P(\lambda)$ satisfies~\eqref{eq:P}, then 
by making use of Prop.~\ref{prop:state-space}, we can also write the rational part in~\eqref{eq:REP} in state-space form as
\begin{equation}\label{eq:R}
R(\lambda)=\sum_{i=0}^{k-1} (A_i-\lambda B_i)f_i(\lambda) + \sum_{i=1}^s  (C_i-\lambda D_i)\, a_i^T(E_i-\lambda F_i)^{-1} b_i
\end{equation}
for some vectors $a_i,b_i \in\mathbb{C}^{\ell_i}$ and the $\ell_i\times \ell_i$ matrices 
\[
E_i = \begin{bmatrix}
\omega_1 & \omega_2 & \cdots & \omega_{\ell_{i-1}} & \omega_{\ell_i} \\
 - z_1 & z_2 \\
& - z_2 & \ddots \\
& & \ddots & z_{\ell_{i-1}} \\
& & & -z_{\ell_{i-1}} & z_{\ell_i} 
\end{bmatrix} \qquad  \text{and} \qquad 
F_i = \begin{bmatrix}
0 & 0 & \cdots & 0 & 0 \\
1  & -1 \\
& 1 & \ddots \\
& & \ddots & -1 \\
& & &1 & -1
\end{bmatrix}.
\]
Next, introduce for $i=1,\ldots, s$ the vector-valued function
\begin{equation}\label{eq:rij}
R_i\colon \C \to \C^{\ell_i}, \qquad R_i(\lambda) = (E_i-\lambda F_i)^{-1}b_i.
\end{equation}
Assuming that $P(\lambda)$ satisfies~\eqref{eq:P} with $f_0(\lambda) =1$ and observing that $(E_i-\lambda F_i)R_i(\lambda) = b_i$ for all $i=1,\ldots,s$, we obtain the linear relation
\[
 \left[\begin{array}{c:c|c}
 \multicolumn{2}{c|}{M-\lambda N} & \begin{matrix} 0 & \cdots & 0 \end{matrix} \\ \hline
	\begin{matrix} -b_1 \\ \vdots \\ -b_s\end{matrix} & 
	\begin{matrix} 0 & \cdots & 0 \\  \vdots & & \vdots \\ 0  & \cdots & 0  \end{matrix} & \begin{matrix} E_1-\lambda F_1 & \\ & \ddots \\ & & E_s-\lambda F_s \end{matrix}
\end{array}\right] 
 \left[\begin{array}{c}
\begin{matrix} f_0(\lambda) \\ \hdashline f_1(\lambda) \\ \vdots \\ f_{k-1}(\lambda) \end{matrix} \\ \hline
\begin{matrix} R_1(\lambda) \\ \vdots \\ R_s(\lambda) \end{matrix}
\end{array}\right] = 0.
\]
% \begin{bv}
% 	This can be prettier by just using one array and filling manually
% \end{bv}
Collecting the basis functions into the single vector
\begin{equation}\label{eq:Psi}
	f(\lambda) = \begin{bmatrix} f_0(\lambda) \\ \vdots \\ f_{k-1}(\lambda) \end{bmatrix}, \quad
\Psi(\lambda) = 
\left[
	\begin{array}{c}
		f(\lambda) \\ \hline R_1(\lambda) \\  \vdots \\ R_s(\lambda)
	\end{array}
\right],
\end{equation}
we arrive at the following result.

\begin{proposition}\label{prop:linear relation big}
Let $\Psi(\lambda)$ be the vector-valued function~\eqref{eq:Psi} with $f_i(\lambda)$ scalar functions satisfying~\eqref{eq:linear relation} such that $f_0(\lambda)=1$ and 
$R_i(\lambda)$ satisfying~\eqref{eq:rij}. If $\lambda\in\mathbb{C}$ is such that $E_i-\lambda F_i$ is invertible for all $i=1,\ldots, s$, then 
\begin{equation}\label{eq:linear relation big}
(\widehat{M}-\lambda\widehat{N})\Psi(\lambda)=0 \qquad \text{with} \quad  \widehat{M}-\lambda\widehat{N} =\left[\begin{array}{c|c}
M-\lambda N & 0 \\ \hline
\begin{matrix} -b & 0 \end{matrix} & E-\lambda F
\end{array}\right],
\end{equation}
where we used
\begin{equation}\label{eq:def_e_EF}
 b = \begin{bmatrix} b_1^T & \cdots & b_s^T \end{bmatrix}^T \qquad \text{and} \quad  \qquad E-\lambda F = \diag(E_1-\lambda F_1, \ldots, E_s-\lambda F_s).
\end{equation}
Furthermore, the pencil $\widehat{M}-\lambda\widehat{N}$ has full row rank for any $\lambda\in\mathbb{C}$ such that  $E_i-\lambda F_i$ is invertible for all $i=1,\ldots, s$.
\end{proposition}
\begin{proof}
The identity~\eqref{eq:linear relation big} was shown above since $f_0(\lambda)=1$. 
The second result is immediate since $M-\lambda N$ has full row rank by assumption and $E-\lambda F$ is only singular when  one of the $E_i-\lambda F_i$ is singular.
\end{proof}

In order to obtain a linearization of~\eqref{eq:R}, we first write it using~\eqref{eq:rij} equivalently as
\begin{align*}
R(\lambda)&=\sum_{i=0}^{k-1} (A_i-\lambda B_i) (f_i(\lambda) \cdot I_n) + \sum_{i=1}^s  (C_i-\lambda D_i)\, (a_i^T R_i(\lambda) \cdot I_n )\\
&=\sum_{i=0}^{k-1} (A_i-\lambda B_i)\, (f_i(\lambda) \otimes I_n) + \sum_{i=1}^s  [a_i^T \otimes (C_i-\lambda D_i)]\, (R_i(\lambda) \otimes I_n).
\end{align*}
Observe that this is a trivial rewriting of scalar multiplications in terms of Kronecker products. However, using the vector $\Psi(\lambda)$ as defined in~\eqref{eq:Psi}, it allows us to express the rational expression in~\eqref{eq:R} as
\begin{align*}
	\renewcommand*{\arraystretch}{.1}
&R(\lambda)=\begin{bmatrix} \begin{array}{c|c} \begin{matrix}  A_0-\lambda B_0 & \cdots & A_{k-1}-\lambda B_{k-1} \end{matrix} & \begin{matrix}  a_1^T\otimes (C_1-\lambda D_1) & \cdots & a_s^T\otimes (C_s-\lambda D_s) \end{matrix} \end{array} \end{bmatrix} \times \\
& \hspace{12cm} (\Psi(\lambda)\otimes I_n).
\end{align*}
Together with Prop.~\ref{prop:linear relation big}, this suggests the following CORK-like linearization of~\eqref{eq:R}.
\begin{definition}{\rm (CORK linearization for AAA rational approximation)}
Let $R(\lambda)$ be the rational approximation \eqref{eq:R} obtained by using the AAA algorithm or the set-valued AAA algorithm.
We define the pencil $\mathcal{L}_R(\lambda)$ as follows
\begin{equation}\label{eq:lin-CORK-rat}
\mathcal{L}_R(\lambda) = 
\left[\begin{array}{c} 
\begin{array}{c|c} \begin{matrix} A_0-\lambda B_0 & \cdots & A_{k-1}-\lambda B_{k-1} \end{matrix} & \begin{matrix} a_1^T\otimes (C_1-\lambda D_1) & \cdots & a_s^T\otimes (C_s-\lambda D_s) \end{matrix} \end{array}
 \\ \hline
\phantom{\Big{(}} (\widehat{M}-\lambda \widehat{N})\otimes I_n \phantom{\Big{(}} 
\end{array}
	\right],
\end{equation}
where the pencil $\widehat{M}-\lambda \widehat{N}$ has been defined in \eqref{eq:linear relation big}.
\end{definition}

The size of $\mathcal{L}_R(\lambda)$ is $(k+\sum_{i=1}^s \ell_i)n$. Fortunately, one can again exploit the Kronecker structure and show that the CORK algorithm can be applied to \eqref{eq:lin-CORK-rat}, as long as the shifts in the shift-and-invert steps of the rational Krylov method  are not poles of the rational interpolants \eqref{eq:rat_approx}.
Furthermore, as a special case of Theorem~\ref{thm:evals_LR_low-rank} below with full-rank matrices, any $\lambda\in\mathbb{C}$ that is not a pole of any of the rational interpolants \eqref{eq:rat_approx} is an eigenvalue of $R(\lambda)$ if and only if it is an eigenvalue of \eqref{eq:lin-CORK-rat}, and their associated right eigenvectors are easily related.
For the set-valued AAA approximation, we have that $E_i-\lambda F_i$ is the same for all $i$, as well as all $b_i$.
As a result, linearization \eqref{eq:lin-CORK-rat} becomes
\begin{equation*}
\mathcal{L}_R(\lambda) = 
\left[
	\begin{array}{cc}
\begin{matrix} A_0-\lambda B_0 & \cdots & A_{k-1}-\lambda B_{k-1} \end{matrix} & \sum_{i=1}^s a_i^T\otimes (C_i-\lambda D_i) \\ \hline
 M - \lambda N & 0 \\
 -b_1 \ \ 0 & E_1 - \lambda F_1
\end{array}
\right],
\end{equation*}
which is of size $(k+\ell_1)n$.

Conclusively, the CORK algorithm can be applied to \eqref{eq:lin-CORK-rat} for computing eigenvalues of $R(\lambda)$ that are not poles of the rational interpolants \eqref{eq:rat_approx} and their associated right eigenvectors.
In practice, we have noticed that this assumption is not very restrictive, since the AAA algorithm tends to place the poles outside the region of interest.

\subsection{Low-rank exploitation}

In several applications, the matrix coefficients of the  nonlinear valued function $G(\lambda)$ in \eqref{eq:G} are usually of low rank.
In this section, we show how the exploitation of these low ranks leads to a linearization of size smaller than that of $\mathcal{L}_R(\lambda)$.
This linearization generalizes the one used in \cite{dopi17,suba11}, which is valid when $P(\lambda)$ in \eqref{eq:P} is expressed using monomials, i.e., $f_i(\lambda)=\lambda^i$, to the more general setting used by CORK.

Suppose that the coefficients of the rational part in \eqref{eq:R} admit the following structure
\begin{equation}\label{eq:low_rank_Z_for_R}
 C_i - \lambda D_i = (\widetilde C_i - \lambda \widetilde D_i) \widetilde Z_i^* \quad \mbox{with} \quad  \widetilde C_i,  \widetilde D_i, \widetilde Z_i \in \C^{n \times k_i}, \quad \mbox{and} \quad  \widetilde Z_i^* \widetilde Z_i = I_{k_i}.
\end{equation}
Observe that this holds trivially for $\widetilde Z_i = I_{k_i}$ but in many problems $k_i$ is potentially much smaller than $n$.
In \cite{suba11}, $\widetilde C_i,  \widetilde D_i, \widetilde Z_i$ are the result of a rank revealing decomposition of $C_i - \lambda D_i$.% The aim is to exploit this low-rank structure in the linearization of $R(\lambda)$.

%
%
% \begin{gather*}
%  \bsym A = \begin{bmatrix} A_0 & \cdots & A_{k-1} \end{bmatrix}, \quad
%  \bsym  B = \begin{bmatrix} B_0 & \cdots & B_{k-1} \end{bmatrix}, \\
%  \bsym { C} =  \begin{bmatrix} a_1^T \otimes \widetilde C_1 & \cdots & a_s^T \otimes \widetilde C_s \end{bmatrix}, \quad
%  \bsym { D} =  \begin{bmatrix} a_1^T \otimes \widetilde D_1 & \cdots & a_s^T \otimes \widetilde D_s \end{bmatrix}, \\
%  \bsym {M} =  M \otimes I_n, \quad
%  \bsym {N} =  N \otimes I_n, \\
%  \bsym {E} =  \diag (E_1 \otimes I_{k_1},  \ldots, E_s \otimes I_{k_s} ), \quad
%  \bsym {F} =  \diag (F_1 \otimes I_{k_1},  \ldots, F_s \otimes I_{k_s} ), \\
%  \bsym { Z} =  \diag (\widetilde{Z}_1, \ldots, \widetilde{Z}_s ), \quad
%   \bsym{e}^T = \begin{bmatrix} e_1^T \otimes I_{k_1} & \cdots & e_1^T \otimes I_{k_s} \end{bmatrix}
% \end{gather*}
% The resulting trimmed matrix pencil becomes
% \[
% \mathcal{L}(\lambda) =\left[\begin{array}{cc}
% \bsym A - \lambda \bsym B & \bsym C - \lambda \bsym D \\
% \bsym M - \lambda \bsym N & \\
% \bsym e \bsym{Z}^* & \bsym E - \lambda \bsym F
% \end{array}\right].
% \]
% It is of size $nk+\ell_1k_1 + \cdots + \ell_s k_s$, compared to the untrimmed size of $n(k+\ell_1 + \cdots +\ell_s)$
%
% \begin{bv}
% TODO See \cite{suba11}, section 7 in \cite{bemm15} and section 4.4 in \cite{vanbeeumen_meerbergen}.
% \end{bv}

Introducing the matrices $\widetilde Z_i$ in the definition of $R(\lambda)$ in \eqref{eq:R}, we obtain
\begin{align*}
R(\lambda)&=\sum_{i=0}^{k-1} (A_i-\lambda B_i) f_i(\lambda) + \sum_{i=1}^s  (a_i^T R_i(\lambda)  ) \cdot (\widetilde C_i-\lambda \widetilde D_i) \widetilde Z_i^* \\
&=\sum_{i=0}^{k-1} (A_i-\lambda B_i)\, (f_i(\lambda) I_n) + \sum_{i=1}^s  [a_i^T \otimes (\widetilde C_i-\lambda \widetilde D_i)]\, [R_i(\lambda) \otimes I_{k_i}] \,  \widetilde Z_i^*,
\end{align*}
where we recall that $R_i(\lambda)=(E_i-\lambda F_i)^{-1}b_i$.
We can therefore write
\[
 R(\lambda) = \left[\begin{array}{cc}
\bsym A - \lambda \bsym B & \bsym C - \lambda \bsym D \end{array}\right] \cdot \bsym{\Psi}(\lambda)
\]
using the matrices
\begin{gather*}
 \bsym A = \begin{bmatrix} A_0 & \cdots & A_{k-1} \end{bmatrix}, \quad
 \bsym  B = \begin{bmatrix} B_0 & \cdots & B_{k-1} \end{bmatrix}, \\
 \bsym { C} =  \begin{bmatrix} a_1^T \otimes \widetilde C_1 & \cdots & a_s^T \otimes \widetilde C_s \end{bmatrix}, \quad
 \bsym { D} =  \begin{bmatrix} a_1^T \otimes \widetilde D_1 & \cdots & a_s^T \otimes \widetilde D_s \end{bmatrix}, \\
 f(\lambda) = \begin{bmatrix} f_0(\lambda) \\ \vdots \\ f_{k-1}(\lambda) \end{bmatrix}, \quad \bsym{\Psi}(\lambda) = \begin{bmatrix} f(\lambda) \otimes I_n \\ (R_1(\lambda) \otimes I_{k_1}) \widetilde Z_1^* \\ \vdots \\ (R_s(\lambda) \otimes I_{k_s}) \widetilde Z_s^* \end{bmatrix}.
\end{gather*}
Denoting by $\bsym O$ a matrix of all zeros (of suitable size), and using
\[
 \bsym {M} =  M \otimes I_n, \quad
 \bsym {N} =  N \otimes I_n, 
\]
where the pencil $M-\lambda N$ is the one in \eqref{eq:linear relation}, we  obtain from $(M-\lambda N) f(\lambda) = 0$ the identity
\[
 \left[\begin{array}{cc}
 \bsym M - \lambda \bsym N &  \bsym O \end{array}\right] \cdot \bsym{\Psi}(\lambda) = \bsym O.
\]
As before $(E_i-\lambda F_i)R_i(\lambda) = b_i$, whence
\[
 [ (E_i - \lambda F_i) \otimes I_{k_i} ] \, [ R_i(\lambda) \otimes I_{k_i} ] \, \widetilde Z_i^* = [b_i \otimes I_{k_i}] \widetilde Z_i^*.
\]
Therefore by assuming again that $f_0(\lambda) \equiv 1$ and introducing
\begin{gather*}
	\bsym {E} =  \diag (E_1 \otimes I_{k_1},  \ldots, E_s \otimes I_{k_s} ), \quad
 \bsym {F} =  \diag (F_1 \otimes I_{k_1},  \ldots, F_s \otimes I_{k_s} ), \\ 
e_1^T = \begin{bmatrix} 1 & 0 & \cdots & 0 \end{bmatrix} \in \R^{k}, \quad \bsym{Z}^* = \begin{bmatrix} -(b_1 \otimes I_{k_1})\widetilde{Z}_1^* \\ \vdots \\ -(b_s \otimes I_{k_s})  \widetilde{Z}_s^*\end{bmatrix} (e_1^T \otimes I_n)
\end{gather*}
we obtain the identity
\[
 \left[\begin{array}{cc}
  \bsym Z^* & \bsym E - \lambda \bsym F \end{array}\right] \cdot \bsym{\Psi}(\lambda) = \bsym O.
\]
%In the matrix above, the vertical bar indicates where the first $n$ columns are located, that is, those that correspond to $f_0(\lambda) I_n$ in $\bsym{\Psi}(\lambda)$.

Putting all the identities from above together, we obtain the following square matrix of size $\widetilde d = nk+ \sum_{i=1}^s \ell_i k_i$:
\begin{equation}\label{eq:trimmed_LR}
\mathcal{\widetilde L}_R(\lambda) \, \bsym{\Psi}(\lambda) = \begin{bmatrix} R(\lambda) \\ \bsym{O} \end{bmatrix} \qquad \text{with }
\mathcal{\widetilde L}_R(\lambda) =\left[\begin{array}{cc}
\bsym A - \lambda \bsym B & \bsym C - \lambda \bsym D \\
\bsym M - \lambda \bsym N & \bsym{O}  \\
 \bsym Z^*  & \bsym E - \lambda \bsym F
\end{array}\right] .
\end{equation}

In Theorem~\ref{thm:evals_LR_low-rank} below we show that, as long as $\lambda$ is not a pole of any of the rational functions $r_i(\lambda)$ in~\eqref{eq:REP}, $\mathcal{\widetilde L}_R(\lambda)$ is indeed a linearization for $R(\lambda)$ in the sense that we can use it to compute the eigenpairs of $R(\lambda)$. Observe that $\widetilde d$ is never larger than $d=n(k+\sum_{i=1}^s \ell_i)$, the size of $\mathcal{L}_R(\lambda)$. Hence, $\mathcal{\widetilde L}_R(\lambda)$ is a trimmed linearization that effectively exploits the low-rank terms in the rational part of $R(\lambda)$. It is also possible to exploit low-rank terms in $P(\lambda)$ as is done in~\cite{bemm15}. However, as this would complicate notation and the gain in size is typically less significant, we do not pursue this here.

Together with Theorem~\ref{thm:evals_LR_low-rank}, we also have in Theorem~\ref{thm:UL_Z} an explicit block-UL factorization of $\mathcal{\widetilde L}_R(\lambda)$. The proof of both these results is fairly standard and is therefore devoted to the appendix---in particular, we refer to similar results in~\cite{suba11} for rational terms $R_i(\lambda)$ with explicit state-space representation, in~\cite{bemm15} for $P(\lambda)$ in CORK  form, and in~\cite{dopi17} for $P(\lambda)$ in companion form combined with explicit state space for $R_i(\lambda)$. However, a compact representation that is a combination of general $P(\lambda)$ in CORK form and rational terms stemming from AAA is new. 

The theorems are stated for a certain permuted version of the columns of $\mathcal{\widetilde L}_R(\lambda)$. By assumption, the $(k-1) \times k$ pencil $M-\lambda N$ has rank $k-1$. Hence, there exists a permutation $\Pi \in \R^{k \times k}$, possibly depending on $\lambda$, such that
\[
 (M - \lambda N) \Pi =: \begin{bmatrix} m_0 - \lambda n_0 & M_1 - \lambda N_1 \end{bmatrix} \qquad \text{with $M_1 - \lambda N_1$ nonsingular.}
\]
Denoting $\bsym{\Pi} = \Pi \otimes I_n$, we can also apply this permutation block-wise to the first $nk$ columns of $\mathcal{\widetilde L}_R(\lambda)$. We then obtain
\begin{equation}\label{eq:trimmed_LR_permuted}
\left[\begin{array}{c|c}
 \bsym A - \lambda \bsym B & \bsym C - \lambda \bsym D \\
 \bsym M - \lambda \bsym N & \bsym{O}  \\
  \bsym Z^*  & \bsym E - \lambda \bsym F
 \end{array}\right] 
 \left[
	 \begin{array}{cc}
  \bsym{\Pi} \\\hline  & I 
 \end{array}
 \right]=
 \left[\begin{array}{cc|c}
  \bsym A_0 - \lambda \bsym B_0 & \bsym A_1 - \lambda \bsym B_1 & \bsym C - \lambda \bsym D \\
  \bsym M_0 - \lambda \bsym N_0 & \bsym M_1 - \lambda \bsym N_1 &  \bsym{O}  \\
   \bsym Z^*_0 &  \bsym Z^*_1 & \bsym E - \lambda \bsym F
  \end{array}\right],
\end{equation}
where $I$ denotes an identity matrix of suitable size, and with
\[
 \bsym M_0 = m_0 \otimes I_n, \quad  \bsym n_0 = n_0 \otimes I_n,   \quad \bsym M_1 = M_1 \otimes I_n, \quad  \bsym N_1 = N_1 \otimes I_n.
\]
The other block matrices are partitioned accordingly. This means $\bsym A_0 = A_j$ and $\bsym B_0 = B_j$ for some $j$ that corresponds to the column that $\bsym{\Pi}$ has permuted to the first position. %In addition $\bsym Z_0^*$ is either zero or the ...

As mentioned above, one of the results is a block UL factorization. Amongst others, it is key for performing efficiently the shift-and-invert steps of the rational Krylov method when computing the eigenvalues of $\widetilde{\mathcal{L}}_R(\lambda)$.
\begin{theorem}\label{thm:UL_Z}
Let $\mathcal{\widetilde L}_R(\lambda)$ be the pencil in~\eqref{eq:trimmed_LR} for the rational matrix $R(\lambda)$ in~\eqref{eq:R} with the low-rank structure~\eqref{eq:low_rank_Z_for_R}. 
If $\mu\in\mathbb{C}$ is such that all $E_1-\mu F_1, \ldots, E_s-\mu F_s$ are nonsingular, then using the block matrices as defined in~\eqref{eq:trimmed_LR_permuted}, the following block-UL decomposition holds:
\[
\mathcal{\widetilde L}_R(\mu) \, \mathcal{P} = \mathcal{U}(\mu)\,\mathcal{L}(\mu),
\]
where (empty blocks are zero and $\rho = {\sum_{i=1}^s \ell_i k_i}$)
\begin{align*}
\mathcal{P} &= \begin{bmatrix} \bsym{\Pi} \\  & I \end{bmatrix} \\
\mathcal{U}(\mu)&=\begin{bmatrix}
I_n & [ \bsym A_1 - \mu \bsym B_1 - \bsym Z_1^* (\bsym C - \mu \bsym D)] \, [\bsym M_1 - \mu \bsym N_1]^{-1} & (\bsym C - \mu \bsym D) (\bsym E - \mu \bsym F)^{-1} \\
 &  I_{(k-1)n}    \\
 &  & I_\rho 
\end{bmatrix}, \\
\mathcal{L}(\mu)&=\begin{bmatrix}
\alpha(\mu)^{-1} \, R(\mu)  \\
\bsym M_0-\mu \bsym N_0 & \bsym M_1-\mu \bsym N_1   \\
\bsym Z_0^* & \bsym Z_1^* & \bsym E-\mu \bsym F
\end{bmatrix}, \quad \alpha(\mu) =  e_1^T \Pi^T f(\mu) \neq 0.
\end{align*}
In addition,
\begin{equation}\label{eq:det_L_R}
 \alpha(\mu)^n \det \mathcal{\widetilde L}_R(\mu) = \det R(\mu) \, (\det ( M_1-\mu  N_1))^{k-1} \, \prod_{i=1}^s (\det (E_i - \mu F_i))^{\ell_i}.
\end{equation}
\end{theorem}
\begin{proof} See appendix~\ref{app:proof_linearization}. \end{proof}

Next, we have the main result for the linearization: the relation of the eigenvalues (and their algebraic and geometric multiplicities) and eigenvectors of the rational matrix $R(\lambda)$ with those of the matrix trimmed pencil $\mathcal{\widetilde L}_R(\lambda)$.
\begin{theorem}\label{thm:evals_LR_low-rank}
Let $\mathcal{\widetilde L}_R(\lambda)$ be the pencil in~\eqref{eq:trimmed_LR} for the rational matrix $R(\lambda)$ in~\eqref{eq:R} with $f_0(\lambda) \equiv 1$ and the low-rank structure~\eqref{eq:low_rank_Z_for_R}.  Let $\lambda_0\in\mathbb{C}$ be such that all $E_1-\lambda_0 F_1, \ldots, E_s-\lambda_0 F_s$ are nonsingular. Denote $\rho = \sum_{i=1}^s \ell_i k_i$.
\begin{enumerate}[(a)]
\item \label{item_a_thm:evals_LR_low-rank} If $x\in\mathbb{C}^{n}$ is an eigenvector of $R(\lambda)$ with eigenvalue $\lambda_0$, then $\bsym{\Psi}(\lambda_0)x \in\mathbb{C}^{kn+\rho}$ is an eigenvector of $\mathcal{\widetilde L}_R(\lambda)$ with eigenvalue $\lambda_0$.
\item \label{item_b_thm:evals_LR_low-rank} If $z\in\mathbb{C}^{kn +\rho}$ is an eigenvector of $\mathcal{\widetilde L}_R(\lambda)$ with eigenvalue $\lambda_0$, then $z=\bsym{\Psi}(\lambda_0) x$ for some eigenvector $x\in\mathbb{C}^n$ of $R(\lambda)$ with eigenvalue $\lambda_0$.
\item \label{item_c_thm:evals_LR_low-rank} The algebraic and geometric multiplicities of $\lambda_0$ as an eigenvalue of $\mathcal{\widetilde L}_R(\lambda)$ and as an eigenvalue of $R(\lambda)$ are the same.
\end{enumerate}
\end{theorem}	
\begin{proof} See appendix~\ref{app:proof_linearization}. \end{proof}

\section{Numerical examples}\label{sec:examples}
This section illustrates the theory with a number of applications.
The AAA algorithm is used to approximate different nonlinear matrix functions and the accuracy of the resulting rational approximations is compared to the accuracy obtained with potential theory (Leja--Bagby points).
The approximations are compared in terms of accuracy and the number of poles (which is equal to the degree plus one) to achieve that accuracy.
% , and the size of the resulting linearization.
We used the rational Krylov method (more specifically, its CORK implementation in \cite{bemm15}) to obtain eigenvalue and eigenvector estimates of the rational approximation.

We compare the following three methods.
\begin{description}
\item[NLEIGS] This is the static variant from \cite{guttel_vanbeeumen}. The rational polynomial is expressed in a basis of rational Newton polynomials. The poles are selected in the branch cut of the nonlinear function and the nodes
are Leja-Bagby points.
\item[AAA-EIGS] This is the rational Krylov method applied to linearization \eqref{eq:lin-CORK}. The rational functions are determined by applying AAA to the $m$ nonlinear functions from \eqref{eq:G} separately.
\item[SV-AAA-EIGS] This is the rational Krylov method applied to linearization \eqref{eq:lin-CORK}. The rational functions are determined using the set-valued AAA approach explain in \S\ref{sec:set-valued AAA}.
\end{description}

In the next section, we review the rational Krylov method that is used in the numerical experiments for finding eigenvalues of the linearizations.
In the numerical experiments, we do not use implicit restarting, i.e., the number of iterations corresponds to the dimension of the Krylov space.

\subsection{The rational Krylov method}

The rational Krylov method, sketched in Algorithm \ref{alg:1}, is a generalization of the shift-and-invert Arnoldi method for solving large-scale generalized eigenvalue problems.
\begin{algorithm}
\caption{Rational Krylov method} 
\label{alg:1}
\begin{algorithmic}[1]
\State Choose vector $v_1$, where $\|v\|_2=1$.
\For{$j=1,2,\hdots$}
\State Choose shift $\sigma_j$.
\State Choose continuation vector $t_j$.
\State Compute $\widehat{v}=(A-\sigma_jB)^{-1}Bw_j$, where $w_j=V_jt_j$.
\State Orthogonalize $\widetilde{v}=v-V_jh_j$, where $h_j=V_j^*\widehat{v}$.
\State  Get new vector $v_{j+1}=\widetilde{v}/h_{j+1,j}$, where $h_{j+1,j}=\|\widetilde{v}\|_2$.
\State Set $V_{j+1}=\begin{bmatrix} V_j & v_{j+1} \end{bmatrix}$.
\EndFor
\end{algorithmic}
\end{algorithm}

At step $j$, Algorithm \ref{alg:1} computes a matrix $V_{j+1}$ whose columns form an orthonormal basis for the rational Krylov subspace
\[
\mathcal{K}_j(A,B,v_1)= \mathrm{span}\{ u_1,u_2,\hdots,u_{j+1} \},
\]
where $u_{i+1}=(A-\sigma_i B)^{-1}B w_i$, for $i=1,\hdots,j$.
Furthermore, the matrix $V_{j+1}$ satisfies the rational
Krylov recurrence relation
\[
AV_{j+1}\underline{H}_j=BV_{j+1}\underline{K}_j,
\]
where $\underline{H}_j\in\mathbb{C}^{(j+1)\times j}$ is an upper Hessenberg matrix whose nonzero entries are the Gram-Schmidt coefficients computed by Algorithm \ref{alg:1}, and
\[
\underline{K}_j=\underline{H}_j \diag(\sigma_1,\hdots,\sigma_j)+\underline{T}_j\in\mathbb{C}^{(j+1)\times j},
\]
where $\underline{T}_j$ is an upper triangular matrix whose $i$th column is the continuation vector $t_i$ appended with some extra zeros.
Note that we use $t_j = e_j$, where $e_j$ is the $j$-th vector of $I$, i.e. we always choose $w_j$ as the iteration vector of the previous step.

Approximations for the eigenvalues and right eigenvectors  of the pencil $A-\lambda B$ are obtained by solving the small generalized eigenvalue problem
\[
K_j s_i = \lambda_i H_js_i,
\]
where $H_j$ and $K_j$ are, respectively, the $j\times j$ upper part of $\underline{H}_j$ and $\underline{K}_j$. 
The pair $(\lambda_i,x_i=V_{j+1}\underline{H}_js_i)$ is referred to as a Ritz pair of the pencil $A-\lambda B$.

\subsection{Gun problem}

The radio-frequency gun cavity problem from the NLEVP collection \cite{betcke_higham} is described by
the following matrix-valued function in $\lambda$
\begin{equation*}
	 A(\lambda) = K - \lambda M + i\sqrt{(\lambda-\sigma_1^2)}W_1 + i\sqrt{(\lambda-\sigma_2^2)}W_2,
\end{equation*}
where $K,M,W_1,W_2 \in \mathbb{R}^{9956 \times 9956}$ are symmetric, positive semi-definite matrices, $\sigma_1=0$, and $\sigma_2=108.8774$.

We accurately approximate the nonlinear part in $A(\lambda)$ in a semi-circle $\Sigma$ in the complex plane, see Figure~\ref{fig:gun_values}.
The function is approximated by a rational function in two different ways.
First, for NLEIGS, a rational polynomial with $31$ poles is used, with poles picked on the branch cut of $\sqrt{(\lambda-\sigma_2^2)}$ on the open interval $(-\infty,\sigma_2]$.
Second, for AAA-EIGS, the AAA test set consists of $500$ random points in the semi-disk combined with $500$ equally distributed points on the boundary.
The resulting poles for both functions $f_j=\sqrt{(\lambda-\sigma_j^2)}$, $j=1,2$, are plotted in Figure~\ref{fig:gun_values}.
Note that not all of the poles are shown, only the ones with a positive real part.

Figure~\ref{fig:gun_approx} shows the approximation error as a function of the number of poles of the rational polynomial for NLEIGS, AAA-EIGS, and SV-AAA-EIGS.
The approximation error is expressed as the relative summed error on the nonlinear functions:
\begin{equation}\label{eq:E_Gun}
	E_f = \sqrt{\sum_{i} \left(\sum_j \frac{f_j(s_i)}{\sum_k f_j^2(s_k)}\right)^2},
\end{equation}
for $j=1,2$, and as the error on the matrix functions:
\begin{equation*}
	E_m = \sqrt{\sum_{i=1}^{m}{\frac{\| A(s_i) - R(s_i) \|^2_1}{\| A(s_i) \|^2_1}}},
\end{equation*}
where the points $s_i, \; i=1,\ldots,1000$ are points of a test set consisting of random points on the bounds and inside the semi-circle.
AAA-EIGS leads to a small reduction in degree compared to NLEIGS.
An important reduction, however, is achieved using the set-valued variant, SV-AAA-EIGS: 
$17$ poles are sufficient for an accuracy of $10^{-13}$ for the AAA set, 
whereas NLEIGS requires $31$ poles.
Both errors $E_f$ and $E_m$ are very comparable.

We then determine eigenvalue and eigenvector estimates of the rational eigenvalue problems, around the central point $\text{Re}(s) = 250^2$, using the rational Krylov method.
We used the same shifts as in~\cite{guttel_vanbeeumen}, i.e. three shifts equally spaced on the real axis and two inside the semi-circle.
Figure~\ref{fig:gun_res} shows residual norms of the five fastest converging Ritz values, as a function of the iteration count, for NLEIGS and SV-AAA-EIGS.
The residual norms are defined as
\begin{equation}
	\label{eq:residue}
	\rho_i = \frac{\| A(\lambda_i)x_i \|_2 }{ \| A(\lambda_i) \|_1 \|x_i\|_2 },
\end{equation}
where $\lambda_i$ is the $i$th eigenvalue estimate, or Ritz value, and $x_i$ is an associated Ritz vector.
In the figure, comparable convergence behavior is observed for both approaches, with slightly less accurate results for NLEIGS.

\begin{center}
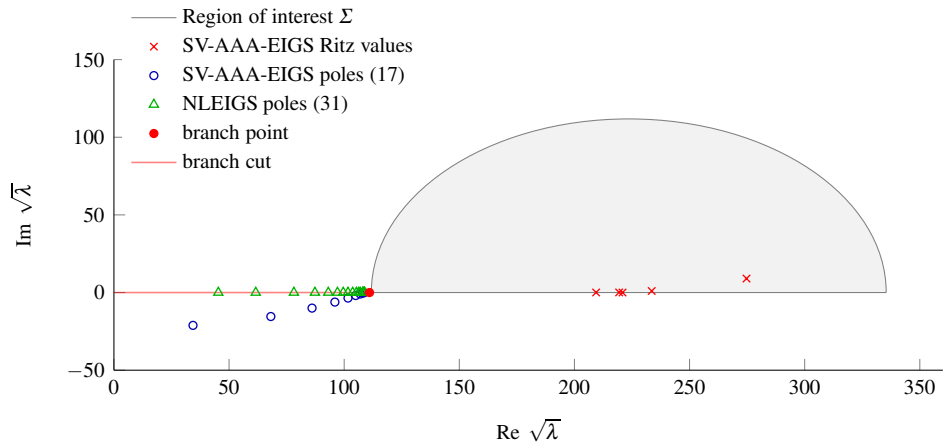
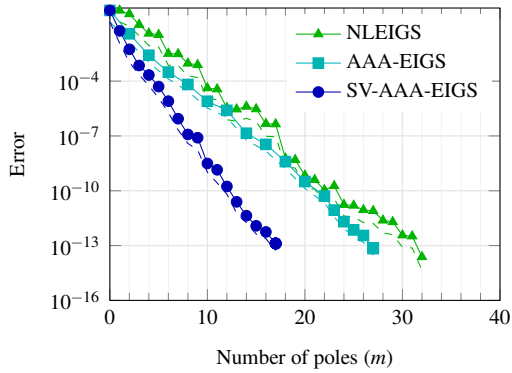
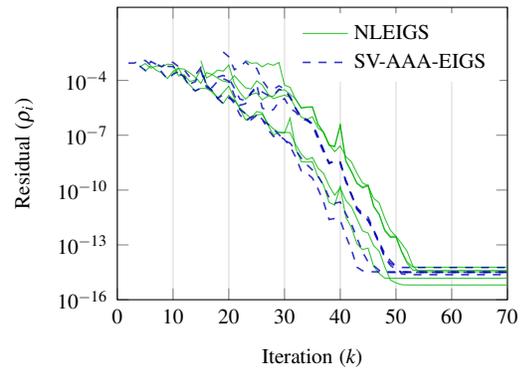
\begin{figure}
		\begin{subfigure}{0.95\textwidth}
			\centering
		\setlength\figurewidth{0.8\textwidth}
		{% This file was created by matlab2tikz.
% Minimal pgfplots version: 1.3
%  http://www.mathworks.com/matlabcentral/fileexchange/22022-matlab2tikz
	\definecolor{green_}{RGB}{0,180,0}%
\definecolor{blue_}{RGB}{0,0,180}%
\begin{tikzpicture}
\begin{axis}[%
width=\figurewidth,
height=0.375\figurewidth,
at={(0\figurewidth,0\figurewidth)},
scale only axis,
unbounded coords=jump,
xmin=0,
xmax=360,
ymin=-50,
ymax=150,
axis x line*=bottom,
axis y line*=left,
xlabel={Re $\: \sqrt{\lambda}$},
ylabel={Im $\: \sqrt{\lambda}$},
label style={font=\footnotesize},
tick label style={font=\footnotesize} ,
legend style={at={(0.01,1.2)},anchor=north west,legend cell align=left,align=left,draw=white!15!black,font=\footnotesize,draw=none,fill opacity=1,text opacity=1}
]
% Bottom domain
\addplot [color=gray,thin,domain=111.8:335.4,samples=1000]  { 0 } ;
\addlegendentry{Region of interest $\Sigma$};

% Ritz values
\addplot [color=red, line width=0.5pt,mark size=2pt,only marks,mark=x,mark options={solid}]
  table[col sep=comma]{tables/gun_values_ritz.txt};
\addlegendentry{SV-AAA-EIGS Ritz values};

% AAA poles
\addplot [color=blue_, line width=0.5pt,mark size=1.5pt,only marks,mark=o,mark options={solid}]
  table[col sep=comma]{tables/gun_values_aaa.txt};
\addlegendentry{SV-AAA-EIGS poles (17)};

% Leja-Bagby poles:
\addplot [color=green_, line width=0.5pt,mark size=2pt,only marks,mark=triangle,mark options={solid}]
  table[col sep=comma]{tables/gun_values_lb.txt};
\addlegendentry{NLEIGS poles (31)};

% The branch point:
\addplot [color=red,opacity=1, line width=0.5pt,mark size=1.5pt,only marks,mark=*,mark options={solid}] coordinates{
	(111,0)
};
\addlegendentry{branch point};
% The branch cut:
\addplot [color=red,opacity=0.5, line width=0.6pt] coordinates{
	(0,0)
	(111,0)
};
\addlegendentry{branch cut};

% The area of interest
\filldraw[fill opacity=0.1,fill=gray,color=gray] (axis cs: 335.4,0) arc[radius =111.8, start angle= 0, end angle= 180];

\end{axis}
\end{tikzpicture}}
		\caption{Poles and Ritz values.}
			 % Poles with $\text{Re}(\lambda)<0$ not shown.}
		\label{fig:gun_values}
	\end{subfigure}

	\hspace{-1cm}
	\vspace{0.4cm}

		\begin{subfigure}{0.49\textwidth}
			\centering
		\setlength\figurewidth{0.73\textwidth}
			{% This file was created by matlab2tikz.
% Minimal pgfplots version: 1.3
%  http://www.mathworks.com/matlabcentral/fileexchange/22022-matlab2tikz
\begin{tikzpicture}

\begin{axis}[%
width=\figurewidth,
height=0.75\figurewidth,
at={(0\figurewidth,0\figurewidth)},
scale only axis,
minor grid style={line width=.1pt, draw=gray!10},
major grid style={line width=.5pt,draw=gray!20},
grid=both,
minor tick num = 4,
restrict x to domain=0:34,
xmin=0,
xmax=40,
xlabel={Number of poles ($m$)},
ymode=log,
ymin=1e-16,
ymax=1e0,
yminorticks=true,
ytick={1e-4,1e-7,1e-10,1e-13,1e-16},
ylabel={Error},
label style={font=\footnotesize},
tick label style={font=\footnotesize} ,
legend style={legend cell align=left,align=left,draw=white!15!black,font=\footnotesize,draw=none,fill opacity=1,text opacity=1}
]
% Leja-Bagby error
\addplot [color=green_,mark size=2pt,solid,mark=triangle*,mark options={solid}]
	table[x = x1, y = y1,col sep=comma]{tables/gun_error.txt};
\addlegendentry{NLEIGS};

\addplot [color=green_,dashed,forget plot]
	table[x = x1, y = y4,col sep=comma]{tables/gun_error.txt};

% AAA error
\addplot [color=blue_light,solid,mark size=2pt,mark=square*,mark options={solid}]
table[x = x2, y = y2,col sep=comma]{tables/gun_error.txt};
\addlegendentry{AAA-EIGS};

\addplot [color=blue_light,dashed,forget plot]
	table[x = x2, y = y5,col sep=comma]{tables/gun_error.txt};

% Set-valued error
\addplot [color=blue_,solid,mark size=2pt,mark=*,mark options={solid}]
table[x = x3, y = y3,col sep=comma]{tables/gun_error.txt};
\addlegendentry{SV-AAA-EIGS};

\addplot [color=blue_,dashed,forget plot]
	table[x = x3, y = y6,col sep=comma]{tables/gun_error.txt};

\end{axis}
\end{tikzpicture}}
			\caption{Error $E_f$ (marked lines) and $E_m$ (dashed lines).}
			\label{fig:gun_approx}
		\end{subfigure}
		\begin{subfigure}{0.49\textwidth}
			\centering
			\setlength\figurewidth{0.73\textwidth}
			{% This file was created by matlab2tikz.
% Minimal pgfplots version: 1.3
%  http://www.mathworks.com/matlabcentral/fileexchange/22022-matlab2tikz
\begin{tikzpicture} 

\begin{axis}[%
width=\figurewidth,
height=0.75\figurewidth,
at={(0\figurewidth,0\figurewidth)},
scale only axis,
xmin=0,
xmax=70,
ymode=log,
minor grid style={line width=.3pt, draw=gray!15},
major grid style={line width=.6pt,draw=gray!20},
xmajorgrids=true,
ymin=1e-16,
ymax=1e0,
yminorticks=true,
xtick distance={10},
ytick={1e-4,1e-7,1e-10,1e-13,1e-16},
xlabel={Iteration ($k$)},
ylabel={Residual ($\rho_i$)},
label style={font=\footnotesize},
tick label style={font=\footnotesize} ,
legend style={legend cell align=left,align=left,draw=white!15!black,font=\footnotesize,draw=none}	
]
% Residues for Leja-Bagby
\foreach \y in {1,2,...,4} {
\addplot[line width=0.4pt,color=green_,solid,forget plot,draw opacity=0.8] table[col sep=comma,x index=0, y index=\y]
    {tables/gun_residue.txt};
}

\addplot[line width=0.4pt,color=green_,solid,draw opacity=0.8] table[col sep=comma,x index=0, y index=5]
    {tables/gun_residue.txt};
% Residues for AAA
\foreach \y in {6,7,...,9} {
\addplot[dashed,line width=0.6pt,color=blue_,draw opacity=0.8,forget plot] table[col sep=comma,x index=0, y index=\y]
    {tables/gun_residue.txt};
}
\addplot[dashed,line width=0.6pt,color=blue_,draw opacity=0.8] table[col sep=comma,x index=0, y index=10]
    {tables/gun_residue.txt};
\legend{NLEIGS,SV-AAA-EIGS}
\end{axis}
\end{tikzpicture}}
			\caption{Residual norms~\eqref{eq:residue}.}
			\label{fig:gun_res}
		\end{subfigure}
		\caption{
			Results for the gun problem. 
		}
\end{figure}
\end{center}

\subsection{Bound states in semiconductor devices}

Determining bound states of a semiconductor device requires the solution of the Schr\"odinger equation,
which, after discretization, leads to a nonlinear eigenvalue problem with the matrix-valued function.
\begin{equation*}
	A(\lambda) = H - \lambda I +  \sum_{j=0}^{80} e^{i\sqrt{\lambda-\alpha_j}}S_j,
\end{equation*}
where $H,S_j \in \mathbb{R}^{16281 \times 16281}$, see \cite{vandenberghe_fischetti,van_beeumen}.
Matrix $H$ is symmetric and matrices $S_j$ have low rank.
This function has $81$ branch points on the real axis at $\lambda=\alpha_j, \; j=0,\ldots,80$, between $-0.19$ and $22.3$, as can be seen in Figure~\ref{fig:states}.
There is a branch cut running from $[-\infty,\alpha_0]$ and one between each branch point.

For approximating the nonlinear functions with Leja--Bagby points, in \cite{vandenberghe_fischetti}, a transformation  is used that removes the branch cut between two predetermined, subsequent branch points, i.e., for $\lambda \in [\alpha_{i-1},\alpha_i]$.
The interpolant based on these Leja--Bagby points is only valid for $\lambda$-values within this interval.
For interval $[\alpha_0,\alpha_1]$, a rational approximation with $50$ poles was used~\cite{vandenberghe_fischetti}.
This corresponds to the green triangular marker in Figure~\ref{fig:states}.

In contrast, using AAA and set-valued AAA, the $81$ nonlinear functions are approximated on the real axis, over multiple branch points, without first transforming the problem.
Figure~\ref{fig:states} shows the resulting number of poles for approximating the nonlinear functions with an accuracy of $10^{-13}$ on a test set of $2000$ equally spaced points, between $\alpha_0$ and the following seven branch points, i.e. for $\lambda \in [\alpha_0,\alpha_i], \; i=1,\ldots,7$.
For example, the third bullet marker indicates the number of poles when the AAA test set runs from $\alpha_0$ to $\alpha_3$, so that it includes branchpoints $\alpha_0,\alpha_1, \alpha_2$ and $\alpha_3$ and the branch cuts in between.
We show the results for AAA and set-valued AAA, the second resulting in a significant reduction of the number of poles for the 81 functions in the example.

The two eigenvalues between $\alpha_0$ and $\alpha_7$ are situated in the first interval, $[\alpha_0,\alpha_1]$, as can be seen in Figure~\ref{fig:states}.
We used the rational Krylov method with five equally spaced shifts from $\alpha_0+\epsilon$ to $\alpha_1-\epsilon$, where $\epsilon = 10^{-2}$.
The convergence behavior, using a SV-AAA-EIGS with 141 poles for interval $[\alpha_0,\alpha_7]$, is shown in Figure~\ref{fig:states_conv}.
We used $\|A(\lambda)x\|_2$ as error measure, to be able to compare the results to those found in \cite{van_beeumen}.
The behavior is comparable to that observed for the static variant of NLEIGS.

\begin{figure}[ht]
\begin{center}
\centering
\setlength\figurewidth{0.8\textwidth}
% Minimal pgfplots version: 1.3
%  http://www.mathworks.com/matlabcentral/fileexchange/22022-matlab2tikz
	\definecolor{green_}{RGB}{0,180,0}%
\definecolor{blue_}{RGB}{0,0,180}%
\begin{tikzpicture}
\begin{axis}[%
width=\figurewidth,
height=0.375\figurewidth,
at={(0\figurewidth,0\figurewidth)},
scale only axis,
unbounded coords=jump,
xmin=-0.3,
xmax=1.5,
ymin=0,
ymax=0,
axis x line = center,
axis y line = none,
restrict x to domain=-0.5:23,
xlabel={Re $\lambda$},
label style={font=\footnotesize},
tick label style={font=\footnotesize} ,
legend style={at={(0.03,0.97)},anchor=north west,legend cell align=left,align=left,draw=white!15!black,font=\footnotesize,draw=none,fill opacity=0,text opacity=1},
x tick style={draw=none},
x tick label style={yshift=0em},
xtick={-0.19,-0.13,0.43,0.88,1.22},
xticklabels={$\alpha_0$,$\alpha_1$,$\alpha_7$,$\alpha_{10}$,$\alpha_{80}$},
extra x ticks={-0.19,0.88,1.22},
extra x tick labels={-0.19,0.88,22.3},
extra x tick style={x tick label style={yshift=1.8em}},
]

\addplot[color=red,opacity=1, line width=0.8pt,mark size=2pt,mark=*,mark options={solid}] table[y = y, x = x,col sep=comma]{tables/states_brpts.txt};

\addplot[color=red,opacity=1,line width=0.8pt] coordinates{
	(-0.5,0)
	(-0.1,0)
};

\addplot[color=red,opacity=1,line width=0.8pt] coordinates{
	(0.88,0)
	(1,0)
};
\label{branchcut}

\addplot[color=white,opacity=1,line width=1pt] coordinates{
	(1,0)
	(1.1,0)
};

\addplot[dotted,color=red,opacity=1,line width=0.8pt] coordinates{
	(1,0)
	(1.1,0)
};

\addplot[color=red,opacity=1,line width=0.8pt] coordinates{
	(1.1,0)
	(1.22,0)
};

\addplot[color=red,opacity=1,line width=0.8pt,mark size=2pt,mark=*,mark options={solid},only marks] coordinates{
	(1.22,0)
};

\label{branchpoint}

\node (br0) at (axis cs:-0.197851945385732,0) {};
\node (br1) at (axis cs:-0.131996471634489,0) {};
\node (br2) at (axis cs:-0.0336677874266508,0) {};
\node (br3) at (axis cs:0.059929216629253,0) {};
\node (br4) at (axis cs:0.115877829357817,0) {};
\node (br5) at (axis cs:0.188433075981491,0) {};
\node (br6) at (axis cs:0.308422313017247,0) {};
\node (br7) at (axis cs:0.428336396950971,0) {};

\end{axis}

\begin{axis}[%
width=\figurewidth,
height=0.375\figurewidth,
at={(0\figurewidth,0\figurewidth)},
scale only axis,
unbounded coords=jump,
xmin=-0.3,
xmax=1.5,
ymin=0,
yshift=-1.5cm,
ymax=0,
axis x line = center,
axis y line = none,
restrict x to domain=-0.5:0.6,
xlabel={Re $\lambda$},
label style={font=\footnotesize},
tick label style={font=\footnotesize} ,
legend style={at={(0.03,0.97)},anchor=north west,legend cell align=left,align=left,draw=white!15!black,font=\footnotesize,draw=none,fill opacity=0,text opacity=1},
x tick style={draw=none},
x tick label style={yshift=0em},
xtick={-0.19,0.5},
xticklabels={$\alpha_0$,$\alpha_1$},
%extra x ticks={-0.19,0.88,1.22},
%extra x tick labels={-0.19,0.88,22.3},
%extra x tick style={x tick label style={yshift=1.8em}},
]

\addplot[color=red,opacity=1, line width=0.8pt,mark size=2pt,mark=*,mark options={solid}] coordinates{
	(-0.19,0)
	(0.5,0)
};

\addplot[color=blue_,opacity=1, line width=0.8pt,mark size=3pt,mark=x,mark options={solid},only marks] coordinates{
	(0.4098,0)
	(0.4946,0)
};

\node (br8) at (axis cs:-0.19,0) {};
\node (br9) at (axis cs:0.5,0) {};

\end{axis}

\begin{axis}[%
width=\figurewidth,
height=0.375\figurewidth,
at={(0\figurewidth,0\figurewidth)},
scale only axis,
unbounded coords=jump,
xmin=-0.3,
xmax = 1.5,
yshift=3cm,%-- CF
ymin=-0.4,
ymax=500,
axis x line = none,
restrict x to domain=-0.5:0.6,
axis y line* = left,
xlabel={$\lambda$},
label style={font=\footnotesize},
tick label style={font=\footnotesize},
ytick={0,100,200,300,400,500},
legend style={at={(0.5,0.8)},anchor=north west,legend cell align=left,align=left,draw=white!15!black,font=\footnotesize,draw=none,fill opacity=1,text opacity=1},
x tick label style={yshift=2em},
xticklabels={$\alpha_1$,$\alpha_2$},
%minor grid style={line width=.3pt, draw=gray!15},
%major grid style={line width=.8pt,draw=gray!25,domain=-0.5:0.6},
%ymajorgrids=true,
%yminorgrids=true,
minor tick num = 4,
ylabel={Number of poles ($m$)},
]

\foreach \x in {0,...,5}
	\addplot[color=gray!25, line width=0.8pt,forget plot] coordinates {
		(-0.3,\x*100)
		(0.56,\x*100)
	};

\foreach \y in {0,...,5}
	\foreach \x in {1,...,5}
		\addplot[color=gray!15, line width=0.3pt,forget plot] coordinates {
			(-0.3,\x*20+\y*100)
			(0.56,\x*20+\y*100)
		};

\addplot[color=blue_light,opacity=1, line width=0.5pt,mark size=2pt,mark=*,mark options={solid}] table[y = y, x = x,col sep=comma]{tables/states_.txt};
\addlegendentry{AAA-EIGS}
\label{AAA_poles}

\addplot[color=blue_,opacity=1, line width=0.5pt,mark size=2pt,mark=*,mark options={solid}] table[y = y, x = x,col sep=comma]{tables/states.txt};
\addlegendentry{SV-AAA-EIGS}
\label{AAA_poles_sv}

\addplot[color=green_,opacity=1,line width=0.5pt,mark size=3pt,mark=triangle*,mark options={solid},only marks] coordinates{
	(-0.132,50)
};
\addlegendentry{NLEIGS}

\addplot[color=blue_,opacity=1, line width=0.8pt,mark size=3pt,mark=x,mark options={solid},only marks] coordinates{
	(-0.5,0)
};

\addplot[color=red,opacity=1, line width=0.8pt,mark size=2pt,mark=*,mark options={solid},only marks] coordinates{
	(-0.4,0)
};

\addplot[color=red,opacity=1,line width=0.8pt] coordinates{
	(-0.5,-0.4)
	(-0.4,-0.4)
};

\addlegendentry{Eigenvalues}
\addlegendentry{Branch point}
\addlegendentry{Branch cut}

\label{Ritz_values}

\node (n1) at (axis cs:-0.131996471634489,206) {};
\node (n2) at (axis cs:-0.0336677874266508,248) {};
\node (n3) at (axis cs:0.059929216629253,304) {};
\node (n4) at (axis cs:0.115877829357817,336) {};
\node (n5) at (axis cs:0.188433075981491,378) {};
\node (n6) at (axis cs:0.308422313017247,403) {};
\node (n7) at (axis cs:0.428336396950971,431) {};

\end{axis}

\draw[color=black,opacity=0.2, line width=1pt]  (br1.center) -- (n1.center);
\draw[color=black,opacity=0.2, line width=1pt]  (br2.center) -- (n2.center);
\draw[color=black,opacity=0.2, line width=1pt]  (br3.center) -- (n3.center);
\draw[color=black,opacity=0.2, line width=1pt]  (br4.center) -- (n4.center);
\draw[color=black,opacity=0.2, line width=1pt]  (br5.center) -- (n5.center);
\draw[color=black,opacity=0.2, line width=1pt]  (br6.center) -- (n6.center);
\draw[color=black,opacity=0.2, line width=1pt]  (br7.center) -- (n7.center);

\draw[color=red,opacity=1., line width=0.2pt,dashed]  (br0.center) -- (br8.center);
\draw[color=red,opacity=1., line width=0.2pt,dashed]  (br1.center) -- (br9.center);

\end{tikzpicture}%
\end{center}
\vspace{-1.5cm}
\caption{
	% Branchpoints (\ref{branchpoint}), branchcuts (\ref{branchcut}) and Ritz values (\ref{Ritz_values}). The top figure shows the number of AAA poles for each function separately approximated (\ref{AAA_poles}) and for set-valued AAA (\ref{AAA_poles_sv}), for all of the functions on the real axis between $\alpha_0$ and $\alpha_i$, $i=1,\ldots,7$, to get a maximum approximation error of $10^{-13}$.
Number of poles for NLEIGS, AAA-EIGS and SV-AAA-EIGS, for seven intervals, for the semiconductor problem.
}
\label{fig:states}
\end{figure}
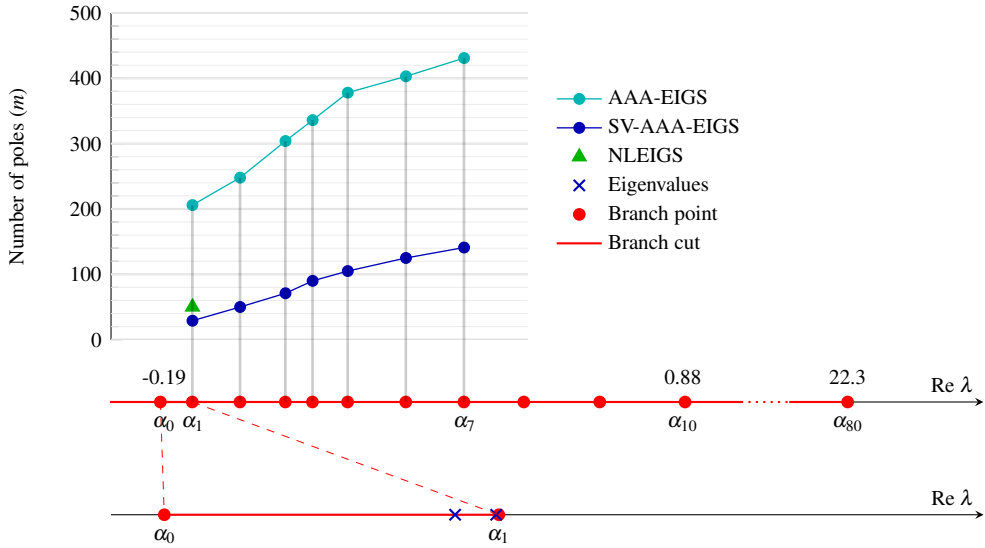

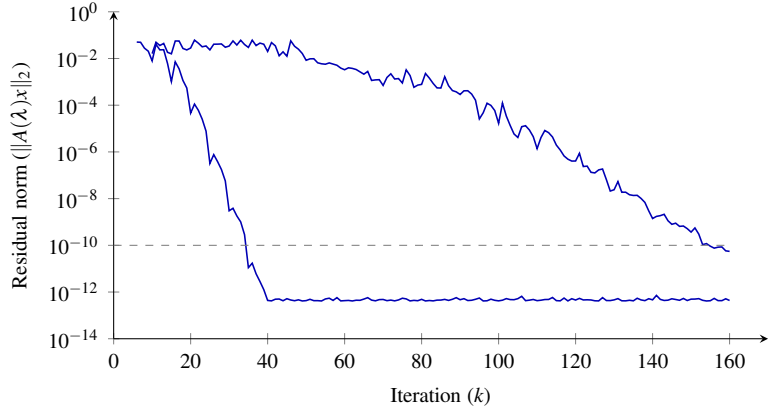
\begin{figure}[ht]
\begin{center}
\centering
\setlength\figurewidth{0.6\textwidth}
% Minimal pgfplots version: 1.3
%  http://www.mathworks.com/matlabcentral/fileexchange/22022-matlab2tikz
\begin{tikzpicture}
\begin{axis}[%
width=\figurewidth,
height=0.5\figurewidth,
at={(0\figurewidth,0\figurewidth)},
scale only axis,
unbounded coords=jump,
ymax=1,
ymin=1e-14,
xmin=0,
xmax=170,
ymode=log,
restrict x to domain=0:160,
xlabel={Iteration ($k$)},
axis x line = bottom,
axis y line = left,
ylabel={Residual norm ($\|A(\lambda)x\|_2$)},
label style={font=\footnotesize},
tick label style={font=\footnotesize},
ytick={1,1e-2,1e-4,1e-6,1e-8,1e-10,1e-12,1e-14},
legend style={at={(0.03,0.97)},anchor=north west,legend cell align=left,align=left,draw=white!15!black,font=\footnotesize,draw=none,fill opacity=0,text opacity=1},
%x tick style={draw=none},
%x tick label style={yshift=0em},
%xtick={-0.19,-0.13,0.43,0.88,1.22},
%xticklabels={$\alpha_0$,$\alpha_1$,$\alpha_7$,$\alpha_{10}$,$\alpha_{80}$},
%extra x ticks={-0.19,0.88,1.22},
%extra x tick labels={-0.19,0.88,22.3},
%extra x tick style={x tick label style={yshift=1.8em}},
]
\addplot[line width=0.6pt,color=blue_,draw opacity=1] table[y = y1, x = x,col sep=comma]{tables/states_residue.txt};
\addplot[line width=0.6pt,color=blue_,draw opacity=1] table[y = y2, x = x,col sep=comma]{tables/states_residue.txt};
\addplot [color=gray,thin,dashed,domain=0:170,samples=100,forget plot]  { 1e-10 } ;
\end{axis}
\end{tikzpicture}%
\end{center}
\caption{
	Convergence behavior of the two Ritz values for set-valued AAA between $\alpha_0$ and $\alpha_7$, with 141 poles, for the semiconductor problem.
}
\label{fig:states_conv}
\end{figure}

\subsection{Sandwich beam}

A beam, consisting of two steel layers surrounding a damping layer, is modeled using the following matrix function
\begin{equation*}
A(\lambda) = K - \lambda^2 M + \frac{G_0 + G_{\infty} (i \lambda \tau)^{\alpha}}{1 + (i \lambda \tau)^{\alpha}}C,
\end{equation*}
with $K,M,C \in \mathbb{R}^{168 \times 168}$ symmetric positive semi-definite matrices \cite{vanbeeumen_meerbergen}.
Here, $G_0 = 350.4$kPa is the static shear modulus, $G_\infty = 3.062$MPa is the asymptotic shear modulus, $\tau = 8.23$ns is the relaxation time and $\alpha=0.675$ a fractional parameter.
Variable $\lambda$ is the angular frequency and we are interested in eigenvalues in the range $\lambda \in [200,30000]$.
%The norm $\|K\|_1 = 1.89 \cdot 10^9$ is much larger than $\|M\|_1 = 4.63 \cdot 10^{-4}$ and $\|C\|_1=5.96 \cdot 10^{-4}$.

We use the AAA algorithm to approximate the function, where a sample set consists of $10^4$ equidistant points within the frequency range.
The algorithm converges with $11$ poles.
The location of the poles is shown in Figure~\ref{fig:beam_vals} and the approximation error, on a random test set of 1000 points $\lambda \in [200,30000]$, is shown in Figure~\ref{fig:beam_err}.
Note that $\text{Re} \; \lambda < -1$ and $\text{Im} \; \lambda > 0$ for all poles, so that we can visualize the poles on a logaritmic axes, with on the negative real axis $-\log_{10}|\text{Re} \; \lambda|$.
Note also that we do not use set-valued AAA, as there is only one non-polynomial function.
% A degree $13$ rational polynomial suffices, while a degree $19$ rational polynomial based on Leja--Bagby points results in the same accuracy.

We then use the rational Krylov method to obtain eigenvalue and eigenvector estimates of $A(\lambda)$.
We use $10$ shifts, $[2,5,10,100,200,210,220,230,240,250]\cdot 100$, i.e., with more shifts located near the end of the interval.
This results in the Ritz values also shown in Figure~\ref{fig:beam_vals}, on the same logaritmic axes, since $\text{Re} \; \lambda > 1$ for all Ritz values.
The residual norms~\eqref{eq:residue} can be seen in Figure~\ref{fig:beam_res}.
Note that because of the large norm of $K$, most residual norm values lie around $10^{-8}$ for $k=1$.
After 50 iterations, the residual norms decrease by a factor $10^{-4}$ for most Ritz values.

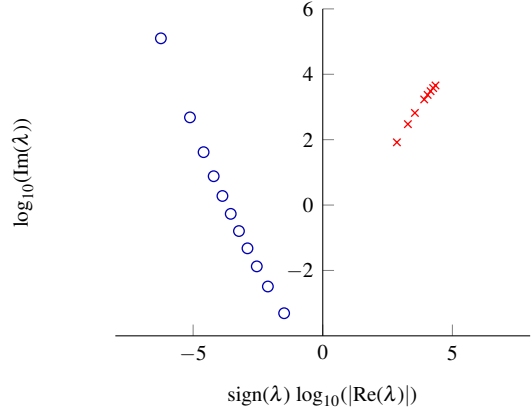
\begin{figure}[ht]
\begin{center}
\centering
\setlength\figurewidth{0.4\textwidth}
% Minimal pgfplots version: 1.3
%  http://www.mathworks.com/matlabcentral/fileexchange/22022-matlab2tikz
\definecolor{mycolor1}{rgb}{0.00000,0.44700,0.74100}%
\definecolor{mycolor2}{rgb}{0.85000,0.32500,0.09800}%
\definecolor{mycolor3}{rgb}{0.92900,0.69400,0.12500}%
\begin{tikzpicture}

\begin{axis}[%
width=0.95092\figurewidth,
height=0.75\figurewidth,
at={(0\figurewidth,0\figurewidth)},
scale only axis,
unbounded coords=jump,
xmin=-8,
xmax=8,
axis y line*=middle,
axis x line*=bottom,
xlabel={sign($\lambda$) $\log_{10}$($|$Re($\lambda$)$|$)},
ymin=-4,
ymax=6,
label style={font=\footnotesize},
tick label style={font=\footnotesize} ,
ylabel={$\log_{10}$(Im($\lambda$))},
legend style={at={(0.8,0.25)},anchor=north west,legend cell align=left,align=left,draw=white!15!black,font=\footnotesize,draw=none,fill opacity=0,text opacity=1}
]

\addplot [color=red, line width=0.5pt,mark size=2pt,only marks,mark=x,mark options={solid}]
  table[col sep=comma]{tables/beam_values_ritz.txt};
\addlegendentry{Ritz values};

\addplot [color=blue_, line width=0.5pt,mark size=2pt,only marks,mark=o,mark options={solid}]
  table[col sep=comma]{tables/beam_values_poles.txt};
\addlegendentry{AAA poles (11)};

\end{axis}
\end{tikzpicture}%
\end{center}
\caption{
	AAA poles and Ritz values for the sandwich beam.
}
\label{fig:beam_vals}
\end{figure}

\begin{figure}[ht]
\begin{center}
\centering
\hspace{-0.2cm}
\begin{subfigure}{0.47\textwidth}
	\centering
	\setlength\figurewidth{0.78\textwidth}
% Minimal pgfplots version: 1.3
%  http://www.mathworks.com/matlabcentral/fileexchange/22022-matlab2tikz
\definecolor{mycolor1}{rgb}{0.00000,0.44700,0.74100}%
\definecolor{mycolor2}{rgb}{0.85000,0.32500,0.09800}%
\begin{tikzpicture}

\begin{axis}[%
width=0.95092\figurewidth,
height=0.75\figurewidth,
at={(0\figurewidth,0\figurewidth)},
restrict x to domain=3:13,
scale only axis,
xmin=0,
xmax=14,
xlabel={Number of poles ($m$)},
minor grid style={line width=.1pt, draw=gray!10},
major grid style={line width=.5pt,draw=gray!20},
minor tick num = 4,
xtick={0,5,10,15},
grid=both,
ymode=log,
ymin=1e-15,
ymax=1e-3,
ytick={1e-1,1e-3,1e-6,1e-9,1e-12,1e-15},
yminorticks=true,
ylabel={Error ($E_f$)},
label style={font=\footnotesize},
tick label style={font=\footnotesize} ,
]
\addplot [color=blue_,solid,mark=*,mark options={solid}] table[col sep=comma,x=x,y=y]
	{tables/beam_error.txt};
\end{axis}
\end{tikzpicture}
	\caption{Approximation error~\eqref{eq:E_Gun} for AAA. \vspace{0.2cm}}
	\label{fig:beam_err}
\end{subfigure}
\hspace{0.2cm}
\begin{subfigure}{0.47\textwidth}
	\centering
	\setlength\figurewidth{0.78\textwidth}
	{% This file was created by matlab2tikz.
% Minimal pgfplots version: 1.3
%  http://www.mathworks.com/matlabcentral/fileexchange/22022-matlab2tikz
\begin{tikzpicture}

\begin{axis}[%
width=0.95092\figurewidth,
height=0.75\figurewidth,
at={(0\figurewidth,0\figurewidth)},
scale only axis,
xmin=0,
xmax=70,
xlabel={Iteration ($k$)},
ylabel={Residual ($\rho_i$)},
ymode=log,
ymin=1e-15,
unbounded coords=jump,
ymax=1e-3,
ytick={1e-3,1e-6,1e-9,1e-12,1e-15},
restrict x to domain=0:80,
%restrict y to domain=1e-16:0.0001,
minor grid style={line width=.3pt, draw=gray!15},
major grid style={line width=.6pt,draw=gray!20},
xmajorgrids=true,
yminorticks=true,
label style={font=\footnotesize},
tick label style={font=\footnotesize} ,
]

\foreach \y in {1,2,...,9} {
\addplot[color=blue_, line width=0.5pt] table[col sep=comma,x index=0, y index=\y]
    {tables/beam_res_in.txt};
}
\label{beam_in}

\foreach \y in {1,2,...,59} {
\addplot[color=blue_, dotted, line width=0.5pt] table[col sep=comma,x index=0, y index=\y]
    {tables/beam_res_out.txt};
}
\label{beam_out}

\end{axis}
\end{tikzpicture}%
}
	\caption{
		Convergence of Ritz values inside (\ref{beam_in}) and outside (\ref{beam_out}) of the interval of interest.
	}
	\label{fig:beam_res}
\end{subfigure}
\end{center}
\caption{
	Approximation error and residuals for the sandwich beam.
}
\end{figure}
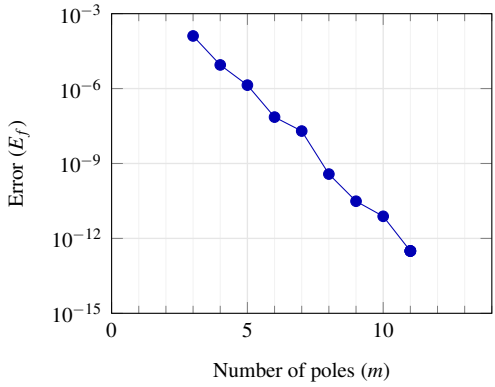
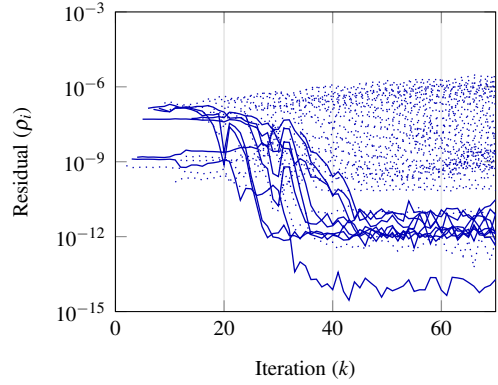

% \begin{figure}[ht]
% \begin{center}
% \centering
% \setlength\figurewidth{0.35\textwidth}
% % \input{figures/beam_allres_lb.tikz}
% % \input{figures/beam_allres_aaa.tikz}
% % \input{figures/beam_allres_aaa_2s.tikz}
% \end{center}
% \caption{
	% \label{fig:beam_res}
% \end{figure}

\subsection{Car cavity problem}

The following model was generated by Dr.~Axel van de Walle from KU Leuven, using a mesh from a Siemens tutorial, and using poro-elastic material properties
from \cite{allard09} and \cite{chazot13}.
The following matrix-valued function describes the nonlinear behavior of the sound pressure inside a car cavity with porous seats:
\begin{equation*}
	A(\lambda) = K_0 + h_K(\lambda)K_1 - \lambda^2(M_0 + h_M(\lambda) M_1),
\end{equation*}
where $K_0,K_1,M_0,M_1 \in \mathbb{R}^{15036 \times 15036}$ are symmetric, positive semidefinite matrices and $\lambda$ is the angular frequency $\lambda=2\pi f$.
The nonlinear functions are given as:
\begin{equation*}
	h_K(\lambda)  =  \frac{\phi}{\alpha(\lambda)}, \quad
	\alpha(\lambda)  =  {\alpha_{\infty} + \frac{\sigma \phi}{i \lambda \rho_0} \sqrt{1+i \lambda \rho_0 \frac{4 \alpha_{\infty}^2 \eta}{\sigma^2 \Lambda^2 \phi^2}}},
\end{equation*}
and
\begin{equation*}
	h_M(\lambda)  =  \phi\left(\gamma-\frac{\gamma-1}{\alpha'(\lambda)}\right), \quad
	\alpha'(\lambda)  =  1 + \frac{8 \eta}{i \lambda \rho_0 \Lambda'^2P_r}\sqrt{1 + i \lambda \rho_0 \frac{\Lambda'^2 P_r}{16 \eta}},
\end{equation*}
with the parameters defined in Table~\ref{tab:car}.
\begin{table}
\caption{Constants of the car cavity model}\label{tab:car}
\begin{center}
\begin{tabular}{|cc|cc|cc|}\hline
	$\alpha_{\infty}$ & $1.7$ & $\sigma$ & $13500 kg m^{-3} s^{-1}$ & $\phi$ & $0.98$   \\\hline
	% & & & & & & & & & \\
$\eta$ & $1.839 \cdot 10^{-5}$ & $\Lambda$ & $80 \cdot 10^{-6} m$ & $\Lambda'$ & $160 \cdot 10^{-6}m$ \\\hline
$\gamma$ & $1.4$ & $\rho_0$ & $1.213$ & $P_r$ & $0.7217$  \\\hline
\end{tabular}
\end{center}
\end{table}
The nonlinear function $h_K$ is unbounded around $\lambda=514i$ and has a branch point around $\lambda=619i$, with a branch cut on the imaginary axis.
The second nonlinear function $h_M$ is unbounded around $\lambda=815i$ and has a branch point at $\lambda=2089i$.
These points and branch cuts are shown in Figure~\ref{fig:car_values_1}.

Since the physical model includes damping, we need to take into account that the eigenvalues have positive imaginary parts.
We chose a test set with $5 \cdot 10^4$ real values, Re$(\lambda) \in [1,300]$, and $5 \cdot 10^4$ random values in the rectangle with corner points $0$ and $300+10^4i$, i.e. with very large imaginary part.
We can see from Figure~\ref{fig:car_values_1} that some AAA poles are picked very close to the singularities of the nonlinear functions.
The degree of the set valued rational approximation is 42 for a tolerance of $10^{-13}$.

In order to compare with potential theory, we then chose a test set over a smaller region, namely the rectangle with corners $0$ and $300+510i$.
In this way, the set of points remains below the first singular point of the function around $514i$.
The Leja--Bagby poles are picked on the imaginary axis, starting at Im$(\lambda) = 514$.
This is shown in Figure~\ref{fig:car_values_2}.
The degree necessary to reach a relative accuracy of $10^{-12}$ on the border of the rectangle was $40$.
For AAA on this same, smaller rectangle, we found that the required approximation accuracy can be obtained with a polynomial of degree 11.
We compared the results obtained by the rational Krylov method for both approximations, where we used $10$ equally spaced shifts on the real axis Re$(\lambda)$, from $1$ to $300$.
Figure~\ref{fig:car_error} shows some of the Ritz values together with the number of Krylov iterations required to reach a residual norm \eqref{eq:residue} below $10^{-12}$.
These Ritz values have low imaginary parts, confirming our choice of test set, and, as for the gun problem, convergence is comparable for both methods.

%We conclude that if no information about the functions would be available, a degree 42 interpolating barycentric form can be set up using AAA.
%It approximates the nonlinear functions well for high imaginary values.
%Otherwise, using the location of branch points and values at infinity, we can limit the imaginary values and use Leja--Bagby points resulting in a Newton polynomial of degree 40.
%However, by limiting the imaginary part of the AAA set in the same way, the degree of the barycentric form is reduced to 11, a significant improvement over the Leja--Bagby case.

\begin{center}
\begin{figure}
	\hspace{-1.5cm}
		\begin{subfigure}{0.65\textwidth}
			\centering
			\setlength\figurewidth{0.8\textwidth}
			{% This file was created by matlab2tikz.
% Minimal pgfplots version: 1.3
%  http://www.mathworks.com/matlabcentral/fileexchange/22022-matlab2tikz
\definecolor{green_}{RGB}{0,180,0}%
\definecolor{blue_}{RGB}{0,0,180}%

\begin{tikzpicture}
\begin{axis}[%
at={(0\figurewidth,0\figurewidth)},
width=\figurewidth,
height=0.5\figurewidth,
scale only axis,
axis lines = middle,
axis line style={-},
unbounded coords=jump,
xmin=-200,
xmax=350,
ymin=0,
ymax=2200,
domain=-200:300,
ytick={514,815},
xtick={-200,-100,0,300},
restrict y to domain=-50:2200,
xlabel={Re $\lambda $},
ylabel={Im $\lambda $},
label style={font=\footnotesize},
tick label style={font=\scriptsize},
extra y ticks={691,2089},
extra y tick labels={$691$,$2089$},
y tick style={draw=none},
%extra y tick labels={$20\sqrt{5}$},    
extra y tick style={y tick label style={right, xshift=0.25em}},
yticklabel pos=right,
ylabel style={at={(ticklabel* cs:1)},anchor=south west},
legend style={at={(-0.2,0.35)},anchor=south west,legend cell align=left,align=left,draw=white!15!black,font=\footnotesize,draw=none,fill opacity=1,text opacity=1}
]

\addplot [color=red, mark size=3pt,only marks,mark=x,mark options={solid}] coordinates {
	(0,514)
	};
\addlegendentry{Singularity};

\addplot [color=red, mark size=3pt,only marks,mark=x,mark options={solid},forget plot] coordinates {
	(0,815.70)
	};

\addplot [color=red,opacity=1, mark size=2pt,only marks,mark=*,mark options={solid}] coordinates {
	(0,691.38)
	};
\addlegendentry{Branch point};

\addplot [color=red, mark size=2pt,only marks,mark=*,mark options={solid},forget plot] coordinates {
	(0,2089)
	};

\addplot [line width=0.8pt,color=red] coordinates {
	(0,691.38)
	(0,2200)
	};
\addlegendentry{Branch cut};

\addplot [color=green_,opacity=0.6, mark size=0.8pt,only marks,mark=*,mark options={solid}]
  table[x = x, y = y,col sep=comma]{tables/car_set_aaa_2.txt};
\addlegendentry{Some AAA test points};

\addplot [color=green_,opacity=0.9, mark size=1.5pt,only marks,mark=o,mark options={solid}]
  table[x = x, y = y,col sep=comma]{tables/car_pol_aaa_2.txt};
\addlegendentry{AAA poles};

\end{axis}
\end{tikzpicture}}
			\caption{Branch points and cuts and values for Re$(\lambda) \in [-200,300]$ and the poles nearest the two singularities.}
		\end{subfigure}
		\hspace{-0.3cm}
		\begin{subfigure}{0.4\textwidth}
			\centering
			\setlength\figurewidth{0.65\textwidth}
			{% This file was created by matlab2tikz.
% Minimal pgfplots version: 1.3
%  http://www.mathworks.com/matlabcentral/fileexchange/22022-matlab2tikz
\definecolor{green_}{RGB}{0,180,0}%
\definecolor{blue_}{RGB}{0,0,180}%

\begin{tikzpicture}
\begin{axis}[%
at={(0\figurewidth,0\figurewidth)},
width=\figurewidth,
height=0.5\figurewidth,
scale only axis,
axis x line = middle,
axis y line = right,
axis line style={-},
unbounded coords=jump,
xmin=-24000,
xmax=1,
ymin=-21000,
ymax=5000,
domain=-200:300,
ytick={-20000,-5000,-1000},
xtick={-23000},
restrict y to domain=-20000:2100,
xlabel={Re $\lambda $},
ylabel={Im $\lambda $},
label style={font=\footnotesize},
tick label style={font=\scriptsize},
%y tick style={draw=none},
%extra y tick labels={$20\sqrt{5}$},    
y tick label style={right, xshift=0.25em},
yticklabel pos=right,
ylabel style={at={(0.72,-0.13)},anchor=east,rotate=-90},
xlabel style={at={(ticklabel* cs:-0.1)},anchor=south west},
extra x ticks={-10000,-5000},
extra x tick labels={-1,-0.5},
extra x tick style={x tick label style={yshift=1.5em}},
]

\addplot [color=green_,opacity=0.9, mark size=1.5pt,only marks,mark=o,mark options={solid}]
  table[x = x, y = y,col sep=comma]{tables/car_pol_aaa_2.txt};

\end{axis}

\begin{axis}[%
xshift=-.08\figurewidth*6/2.5,%-- CF
yshift=-.3\figurewidth*1.2/2.5,
width=.5\figurewidth,
height=.5\figurewidth,
axis y line*=right,
axis x line*=top,
scaled y ticks = false,
scaled x ticks = false,
xmin=-205000,
xmax=-195000,
ymin=-1670000,
ymax=-1530000,
restrict y to domain=-1650000:-1550000,
ytick={-1593000},
xtick={-200800},
xticklabels={-2.0e5},
yticklabels={-1.6e6},
tick label style={font=\scriptsize},
legend style={at={(0.9,0.03)},anchor=south west,legend cell align=left,align=left,draw=white!15!black,font=\footnotesize,draw=none,fill opacity=1,text opacity=1}
]

\addplot [color=green_,opacity=0.9, mark size=1.5pt,only marks,mark=o,mark options={solid}]
  table[x = x, y = y,col sep=comma]{tables/car_pol_aaa_2.txt};

\end{axis}

\begin{axis}[%
xshift=5.5\figurewidth*0.5/2.5,
yshift=.6\figurewidth*3/2.5,
width=.5\figurewidth,
height=.5\figurewidth,
axis y line*=left,
axis x line*=bottom,
scaled y ticks = false,
scaled x ticks = false,
xmin=320000,
xmax=420000,
ymin=1300000,
ymax=2300000,
restrict y to domain=1300000:2300000,
xtick={376000},
ytick={1790000},
xticklabels={3.8e5},
yticklabels={1.8e6},
tick label style={font=\scriptsize},
legend style={at={(0.9,0.03)},anchor=south west,legend cell align=left,align=left,draw=white!15!black,font=\footnotesize,draw=none,fill opacity=1,text opacity=1}
]

\addplot [color=green_,opacity=0.9, mark size=1.5pt,only marks,mark=o,mark options={solid}]
  table[x = x, y = y,col sep=comma]{tables/car_pol_aaa_2.txt};

\end{axis}

\end{tikzpicture}}

			\vspace{1cm}

			\caption{All of the AAA poles.}
		\end{subfigure}
	\caption{Some test points and poles of AAA for the car cavity problem for large imaginary values.}
	\label{fig:car_values_1}
\end{figure}
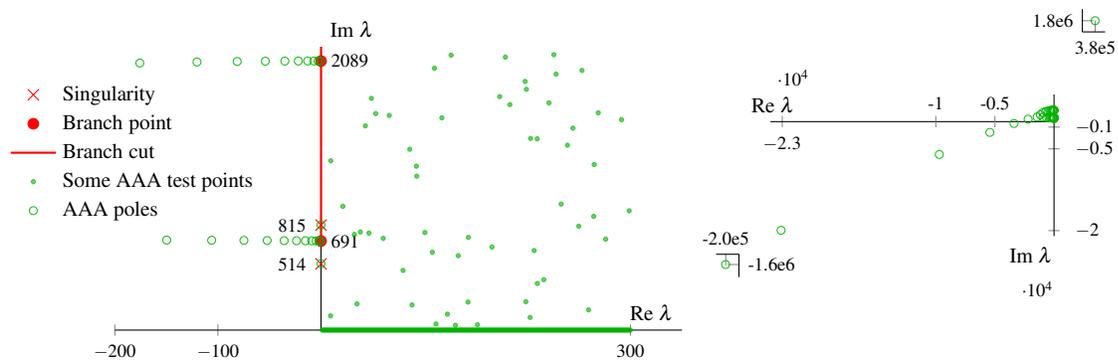
\end{center}

\hspace{0.2cm}

\begin{center}
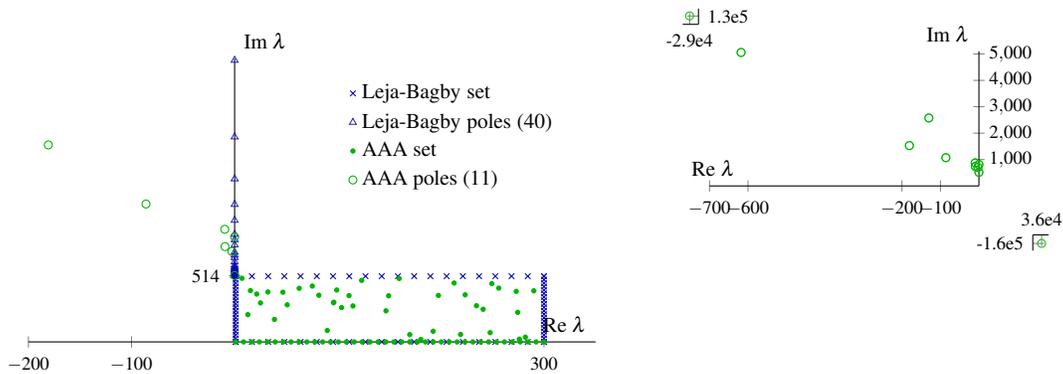
\begin{figure}
	\hspace{-1.4cm}
		\begin{subfigure}{0.65\textwidth}
		 \centering
		 \setlength\figurewidth{0.8\textwidth}
		 {% This file was created by matlab2tikz.
% Minimal pgfplots version: 1.3
%  http://www.mathworks.com/matlabcentral/fileexchange/22022-matlab2tikz
\definecolor{green_}{RGB}{0,180,0}%
\definecolor{blue_}{RGB}{0,0,180}%

\begin{tikzpicture}
\begin{axis}[%
at={(0\figurewidth,0\figurewidth)},
width=\figurewidth,
height=0.5\figurewidth,
scale only axis,
axis lines = middle,
axis line style={-},
unbounded coords=jump,
xmin=-200,
xmax=350,
ymin=0,
ymax=2200,
domain=-200:300,
ytick={514},
xtick={-200,-100,0,300},
restrict y to domain=-50:2200,
xlabel={Re $\lambda $},
ylabel={Im $\lambda $},
label style={font=\footnotesize},
tick label style={font=\scriptsize},
ylabel style={at={(ticklabel* cs:1)},anchor=south west},
legend style={at={(0.55,0.5)},anchor=south west,legend cell align=left,align=left,draw=white!15!black,font=\footnotesize,draw=none,fill opacity=1,text opacity=1}
]

\addplot [color=blue_, mark size=1.5pt,only marks,mark=x]
  table[x = x, y = y,col sep=comma]{tables/car_set_lb.txt};
\addlegendentry{Leja-Bagby set};

\addplot [color=blue_, opacity=0.8, mark size=1.5pt,only marks,mark=triangle,mark options={solid}]
  table[x = x, y = y,col sep=comma]{tables/car_pol_lb.txt};
\addlegendentry{Leja-Bagby poles (40)};

\addplot [color=green_,opacity=1, mark size=0.8pt,only marks,mark=*,mark options={solid}]
  table[x = x, y = y,col sep=comma]{tables/car_set_aaa_1.txt};
\addlegendentry{AAA set};

\addplot [color=green_, mark size=1.5pt,only marks,mark=o,mark options={solid}]
  table[x = x, y = y,col sep=comma]{tables/car_pol_aaa_1.txt};
\addlegendentry{AAA poles (11)};

\end{axis}
\end{tikzpicture}}
		 \caption{Leja--Bagby and AAA set and poles.}
		\end{subfigure}
		\hspace{-0.8cm}
		\begin{subfigure}{0.38\textwidth}
		 \centering
		 \setlength\figurewidth{0.65\textwidth}
		 {% This file was created by matlab2tikz.
% Minimal pgfplots version: 1.3
%  http://www.mathworks.com/matlabcentral/fileexchange/22022-matlab2tikz
\definecolor{green_}{RGB}{0,180,0}%
\definecolor{blue_}{RGB}{0,0,180}%

\begin{tikzpicture}
\begin{axis}[%
at={(0\figurewidth,0\figurewidth)},
width=\figurewidth,
height=0.5\figurewidth,
scale only axis,
axis lines = middle,
axis line style={-},
unbounded coords=jump,
xmin=-700,
xmax=0,
ymin=0,
ymax=5100,
domain=-200:300,
ytick={0,1000,2000,3000,4000,5000},
xtick={-700,-600,-200,-100},
restrict y to domain=0:5100,
xlabel={Re $\lambda $},
ylabel={Im $\lambda $},
label style={font=\footnotesize},
tick label style={font=\scriptsize},
%y tick style={draw=none},
%extra y tick labels={$20\sqrt{5}$},    
y tick label style={right, xshift=0.25em},
yticklabel pos=right,
ylabel style={at={(ticklabel* cs:1)},anchor=south east},
xlabel style={at={(ticklabel* cs:-0.1)},anchor=south west},
]

\addplot [color=green_,opacity=0.9, mark size=1.5pt,only marks,mark=o,mark options={solid}]
  table[x = x, y = y,col sep=comma]{tables/car_pol_aaa_1.txt};

\addplot [color=green_,opacity=0.9, mark size=1.5pt,only marks,mark=o,mark options={solid}]
  table[x = x, y = y,col sep=comma]{tables/car_pol_aaa_1.txt};

\end{axis}

\begin{axis}[%
xshift=.5\figurewidth*6/2.5,%-- CF
yshift=-.5\figurewidth*1.2/2.5,
width=.5\figurewidth,
height=.5\figurewidth,
axis y line*=left,
axis x line*=top,
scaled y ticks = false,
scaled x ticks = false,
xmin=35000,
xmax=36000,
ymax=-150000,
ymin=-170000,
restrict y to domain=-170000:-150000,
ytick={-160900},
xtick={35560},
xticklabels={3.6e4},
yticklabels={-1.6e5},
tick label style={font=\scriptsize},
legend style={at={(0.9,0.03)},anchor=south west,legend cell align=left,align=left,draw=white!15!black,font=\footnotesize,draw=none,fill opacity=1,text opacity=1}
]

\addplot [color=green_,opacity=0.9, mark size=1.5pt,only marks,mark=o,mark options={solid}]
  table[x = x, y = y,col sep=comma]{tables/car_pol_aaa_1.txt};

\end{axis}

\begin{axis}[%
xshift=-.5\figurewidth*0.5/2.5,
yshift=.5\figurewidth*3/2.5,
width=.5\figurewidth,
height=.5\figurewidth,
axis y line*=right,
axis x line*=bottom,
scaled y ticks = false,
scaled x ticks = false,
xmin=-29300,
xmax=-28300,
ymin=123000,
ymax=133000,
restrict y to domain=123000:133000,
ytick={128300},
xtick={-28860},
xticklabels={-2.9e4},
yticklabels={1.3e5},
tick label style={font=\scriptsize},
legend style={at={(0.9,0.03)},anchor=south west,legend cell align=left,align=left,draw=white!15!black,font=\footnotesize,draw=none,fill opacity=1,text opacity=1}
]

\addplot [color=green_,opacity=0.9, mark size=1.5pt,only marks,mark=o,mark options={solid}]
  table[x = x, y = y,col sep=comma]{tables/car_pol_aaa_1.txt};

\end{axis}

\end{tikzpicture}}

		 \vspace{1.5cm}

		 \caption{All of the AAA poles.}
		\end{subfigure}
		\hspace{1cm}

		\caption{Leja--Bagby and AAA sets for the car cavity problem, up to $514i$, and accompanying poles.}
%		\label{fig:car_values_2}
\label{fig:car_values_2}
\end{figure}
\end{center}

\begin{figure}
\begin{center}
\centering
\setlength\figurewidth{0.4\textwidth}
% Minimal pgfplots version: 1.3
%  http://www.mathworks.com/matlabcentral/fileexchange/22022-matlab2tikz
\definecolor{green_}{RGB}{0,180,0}%
\definecolor{blue_}{RGB}{0,0,180}%
\begin{tikzpicture}
\begin{axis}[%
width=\figurewidth,
height=0.375\figurewidth,
at={(0\figurewidth,0\figurewidth)},
scale only axis,
unbounded coords=jump,
xmin=0,
xmax=300,
ymin=0,
ymax=10,
x=0.45,
xtick={50,100,150,200,250,300},
y=12,
axis x line*=bottom,
axis y line*=left,
xlabel={Re $\lambda $},
ylabel={Im $\lambda $},
label style={font=\footnotesize},
tick label style={font=\footnotesize} ,
]
%\addplot [color=green_, line width=0.5pt,mark size=2pt,only marks,mark=triangle,mark options={solid}]
%  table[col sep=comma]{figures/car_values_table_1.txt};
%\addlegendentry{Leja-Bagby poles};

%\addplot [color=blue_, line width=0.5pt,mark size=1.5pt,only marks,mark=o,mark options={solid}]
%  table[col sep=comma]{figures/car_values_table_2.txt};
%\addlegendentry{AAA poles};

\addplot [color=red, line width=0.5pt,mark size=2pt,only marks,mark=x]
  table[x = x_ritz, y = y_ritz,col sep=comma]{tables/car_conv.txt};
%\addlegendentry{Ritz values};

\node (r1) at (axis cs:98.7088329166687,1.1293005421985) {};
\node (r2) at (axis cs:177.786377894053,1.59357424285004) {};
\node (r3) at (axis cs:141.504725959656,1.55405284532905) {};
\node (r4) at (axis cs:130.189303441725,2.15956413192157) {};
\node (r5) at (axis cs:79.5615184237731,1.51534562200036) {};
\node (r6) at (axis cs:141.290058007675,4.28640030757782) {};

%\addplot[color=gray,opacity=0.5, line width=0.5pt,domain=0:2.4771,samples=1000]  { 0 } ;
%\addplot[color=gray,opacity=0.5, line width=0.5pt]  coordinates {
%	(2.4771,0)
%	(2.4771,1)
%};
\end{axis}

\begin{axis}[%
x=3,
y=20,
xmin=0,
xmax=40,%--- CF
xshift=6cm,%-- CF
width=8cm,
axis y line*=left,
axis x line*=bottom,
ymin=0.4, 
ymax=6.5,
yticklabels = {,,},
xtick={25,30,35,40},
minor xtick={26,27,28,29,31,32,33,34,36,37,38,39},
minor grid style={line width=.1pt, draw=gray!10},
major grid style={line width=.6pt,draw=gray!20},
xmajorgrids=true,
xminorgrids=true,
xlabel={\# iterations $>10^{-12}$},
legend style={at={(0.9,0.03)},anchor=south west,legend cell align=left,align=left,draw=white!15!black,font=\footnotesize,draw=none,fill opacity=1,text opacity=1}
]

\addplot[xbar,color=green_,fill=green_,bar width=3] table[y = ii, x = k_lb,col sep=comma]{tables/car_conv.txt};
\addlegendentry{Leja-Bagby};
\addplot[xbar,color=blue_,fill=blue_,bar width=3] table[y = ii2, x = k_aaa,col sep=comma]{tables/car_conv.txt};
\addlegendentry{AAA};

%\addplot[red,fill=red] coordinates
%{(0.5,0.75) (0,0.25) (0,0.25) (0,0.25) (0,0.25) (0,0.25)};

\node (v1) at (axis cs:0,1) {};
\node (v2) at (axis cs:0,2) {};
\node (v3) at (axis cs:0,3) {};
\node (v4) at (axis cs:0,4) {};
\node (v5) at (axis cs:0,5) {};
\node (v6) at (axis cs:0,6) {};

\end{axis}

\draw[color=red,opacity=0.2, line width=0.8pt]  (r1.center) -- (v1.center);
\draw[color=red,opacity=0.2, line width=0.8pt]  (r2.center) -- (v2.center);
\draw[color=red,opacity=0.2, line width=0.8pt]  (r3.center) -- (v3.center);
\draw[color=red,opacity=0.2, line width=0.8pt]  (r4.center) -- (v4.center);
\draw[color=red,opacity=0.2, line width=0.8pt]  (r5.center) -- (v5.center);
\draw[color=red,opacity=0.2, line width=0.8pt]  (r6.center) -- (v6.center);

\end{tikzpicture}%
\end{center}
\caption{
	Ritz values and number of iterations to reach tolerance $10^{-12}$ for the car cavity problem.
}
\label{fig:car_error}
\end{figure}
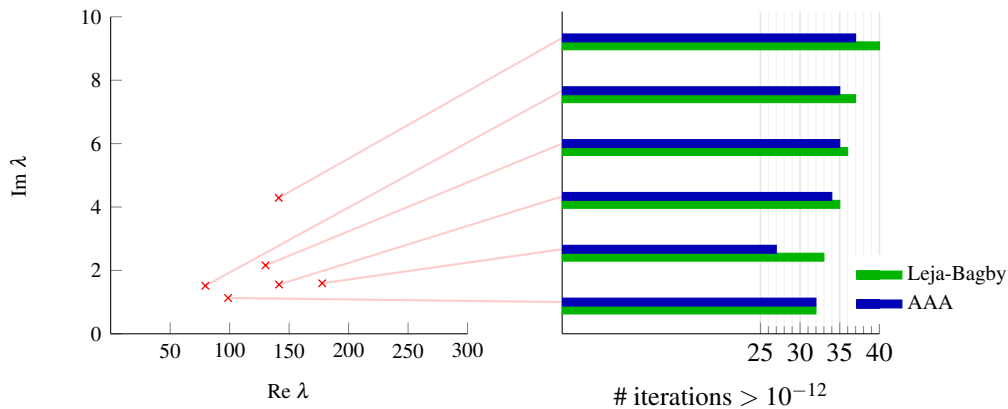
% \FloatBarrier

\section{Conclusions}\label{sec:conlusions}

We have proposed a method for solving the nonlinear eigenvalue problem by first approximating the associated nonlinear matrix valued function using the AAA algorithm.
  This approximation is embedded in a state space representation, which leads to a strong linearization that can be solved using the Compact Rational Krylov method.
  We presented two versions: one that approximates each function separately and then the set-valued version that approximates all functions together.
  The latter version is very competitive with NLEIGS in terms of degree of the rational approximation, and in all our tests,
  AAA requires less poles than an approximation using potential theory, even for a problem with eighty different functions.
  The main advantage of the method is the fully automatic procedure in the determination of the rational approximations.
  Although we did not try solving a problem described by an extremely large number of nonlinear functions, i.e., of the order of the size of the matrix, we expect that the construction of the AAA approximation may
  become prohibitive, when the number of nonlinear functions is extremely high, which is a disadvantage that NLEIGS does not share.

%\ifnum 1=1
\section*{Acknowledgment}
The authors thank Dr.~Axel van de Walle for providing us with the car cavity model that he generated for us using a model from a Virtual Lab tutorial (formerly LMS International, currently Siemens Industry Software).

\bibliographystyle{IMANUM-BIB}
\bibliography{stringlong,library}

\appendix

\section*{Appendix A.\ Set-valued AAA implementation}\label{app:setAAA}

\lstset{language=Matlab,%
	%basicstyle=\color{red},
	breaklines=true,%
	morekeywords={matlab2tikz},
	% keywordstyle=\color{blue},%
	% morekeywords=[2]{1}, keywordstyle=[2]{\color{black}},
	identifierstyle=\color{black},%
	% stringstyle=\color{mylilas},
	% commentstyle=\color{mygreen},%
	showstringspaces=false,%without this there will be a symbol in the places where there is a space
	%numbers=left,%
	%numberstyle={\tiny \color{black}},% size of the numbers
	%numbersep=9pt, % this defines how far the numbers are from the text
	basicstyle=\footnotesize,
	% emph=[1]{for,end,break},emphstyle=[1]\color{red}, %some words to emphasise
	%emph=[2]{word1,word2}, emphstyle=[2]{style},    
}
\begin{lstlisting}
function [r, pol, res, zer, z, f, w, errvec] = aaa2(F, Z, tol, mmax)
%AAA   Computes a AAA rational approximation where F can have multiple
%outputs
%
% Input: F   = matrix of data values with the number of sample points in the first 
%              dimension and the number of functions in the second dimension
%        Z   = vector of sample points
%        tol = relative tolerance [1e-13]
%        mmax = maximum degree [100]
%
% Output: r = AAA approximate function
%         pol, res, zer = poles, residues and zeros
%         z,f,w = interpolation points, function values and weights
%         errvec = errors in each step

% Scale the functions
normF = max(abs(F),[],1);
F = bsxfun(@rdivide,F,normF);

% Left scaling matrix:
SF = spdiags(F, 0:-M:-M*(nF-1), M*nF, M);

% Initialize values
F = F(:);
R = mean(F);
errvec = zeros(mmax,1);
z = zeros(mmax,1);
f = zeros(mmax,nF);
ind = zeros(mmax,nF);
H = zeros(mmax,mmax-1);
S = zeros(mmax,mmax-1);

Q = zeros(M*nF,0);
C = zeros(M,0);

% AAA iteration:
for m = 1:mmax
    [errvec(m),loc] = max(abs(F-R));               % Select next support point 
    if ( errvec(m) <= tol )                        % where error is largest
        m = m-1;
        break
    end
    loc = mod(loc,M);
    ind(m,:) =  loc + (M*(loc==0):M:(nF-1+(loc==0))*M);  % Get indices of the z_i
    z(m) = Z(ind(m,1));                           % Add interpolation point
    f(m,:) = F(ind(m,:));                         % Add function values
    C(:,end+1) = 1./(Z - z(m));                   % Column of the Cauchy matrix
    C(ind(1:m,1),m) = 0;                          % Set the selected elements to 0
    
    v = C(:,m)*f(m,:);                            % Compute the next vector
    v = SF*C(:,m)-v(:);
    
    % Update H and S to compensate for the removal of the rows
    q = Q(ind(m,:),1:m-1);
    q = q*S(1:m-1,1:m-1);
    Si = chol(eye(m-1,m-1)-q'*q);
    H(1:m-1,1:m-1) = Si*H(1:m-1,1:m-1);
    S(1:m-1,1:m-1) = S(1:m-1,1:m-1)/Si;
    S(m,m) = 1;
    Q(ind(1:m,:),:) = 0;
    
    nv = norm(v);
    H(1:m-1,m) = Q'*v;
    H(1:m-1,m) = S(1:m-1,1:m-1)'*H(1:m-1,m);
    HH = S(1:m-1,1:m-1)*H(1:m-1,m);
    v = v-(Q*HH);
    H(m,m) = norm(v);
    % Reorthoganlization is necessary for higher precision
    it = 0;
    while (it < 3) && (H(m,m) < 1/sqrt(2)*nv)
        h_new = S(1:m-1,1:m-1)'*(Q'*v);
        v = v - Q*(S(1:m-1,1:m-1)*h_new);
        H(1:m-1,m) = H(1:m-1,m) + h_new;
        nv = H(m,m);
        H(m,m) = norm(v);
        it = it+1;
    end
    v = v/H(m,m);
    
    % Add v
    Q(:,end+1) = v;
    
    % Solve small least squares problem with H
    [~,~,V] = svd(H(1:m,1:m));
    w = V(:,end);
    
    % Get the rational approximation
    N = C*bsxfun(@times,w,f(1:m,:));       % Numerator
    D = C*bsxfun(@times,w,ones(m,nF));     % Denominator
    R = N(:)./D(:);
    R(ind(1:m,:)) = F(ind(1:m,:));
end

f = f(1:m,:);
w = w(1:m);
z = z(1:m);

% Scale function values back
f = bsxfun(@times,f,normF);

% Note: When M == 2, one weight is zero and r is constant.
% To obtain a good approximation, interpolate in both sample points.
if ( M == 2 )
    z = Z;
    f = F;
    w = [1; -1];       % Only pole at infinity.
    w = w/norm(w);     % Impose norm(w) = 1 for consistency.
    errvec(2) = 0;
end

% Remove support points with zero weight:
I = find(w == 0);
z(I) = [];
w(I) = [];
f(I,:) = [];

% Construct function handle:
r = @(zz) reval(zz, z, f, w);

% Compute poles, residues and zeros:
[pol, res, zer] = prz(r, z, f, w);

end
\end{lstlisting}

\section*{Appendix B.\ Proofs of Theorems \ref{thm:UL_Z} and \ref{thm:evals_LR_low-rank}}\label{app:proof_linearization}

The following simple lemma turns out be useful. 
\begin{lemma}\label{lem:schur_R}
Let $G_{22} \in \C^{m \times m}$ and $X_1 \in \C^{n \times n}$ be invertible matrices that satisfy 
\begin{equation}\label{eq:G_X_Y}
 \begin{bmatrix} G_{11} & G_{12} \\ G_{21} & G_{22} \end{bmatrix} \begin{bmatrix} X_1 \\ X_2  \end{bmatrix} = \begin{bmatrix} Y_1 \\ O \end{bmatrix},
\end{equation}
then the following block-UL decomposition holds:
\begin{equation}\label{eq:G_UL}
 \begin{bmatrix} G_{11} & G_{12} \\ G_{21} & G_{22} \end{bmatrix} = 
  \begin{bmatrix} I_n & G_{12} G_{22}^{-1} \\ O & I_{m} \end{bmatrix}
  \begin{bmatrix} Y_1 X_1^{-1} & O \\ G_{21} & G_{22} \end{bmatrix}.
\end{equation}
\end{lemma}
\begin{proof}
After block elimination of the square matrix on the left-hand size of~\eqref{eq:G_UL}, we only need to show that its Schur complement $S = G_{11} - G_{12} G_{22}^{-1}  G_{21}$ equals $Y_1 X_1^{-1}$. But this follows directly from the identities $G_{11} = (Y_1 - G_{12} X_2)X_1^{-1}$ and $X_2 = - G_{22}^{-1}G_{21} X_1$ that are implied by~\eqref{eq:G_X_Y}. 
\end{proof}

\begin{proof}[Proof of Theorem~\ref{thm:UL_Z}]
	Identity~\eqref{eq:det_L_R} follows from direct manipulation together with $\det (A \otimes I_n) = (\det(A))^n$. To show the block-UL decomposition, recall from~\eqref{eq:trimmed_LR} and~\eqref{eq:trimmed_LR_permuted} that
\[
 (\mathcal{\widetilde L}_R(\lambda) \mathcal{P}) \, (\mathcal{P}^T\bsym{\Psi}(\lambda)) = \begin{bmatrix} R(\lambda) \\ \hline \bsym{O} \end{bmatrix} \text{ with }
 \mathcal{\widetilde L}_R(\lambda) \mathcal{P} = \left[\begin{array}{c|cc}
 \bsym A_0 - \lambda \bsym B_0 & \bsym A_1 - \lambda \bsym B_1 & \bsym C - \lambda \bsym D \\\hline
 \bsym M_0 - \lambda \bsym N_0 & \bsym M_1 - \lambda \bsym N_1 &  \bsym{O}  \\
 \bsym Z^*_0 &  \bsym Z^*_1 & \bsym E - \lambda \bsym F
 \end{array}\right].
\]
The vertical and horizontal lines indicate compatible block partitioning. The corresponding partitioning for $\mathcal{P}^T\bsym{\Psi}(\lambda)$ satisfies
\[
 \mathcal{P}^T \bsym{\Psi}(\lambda) =  \begin{bmatrix} p_1 I_n \\ \hline p_2 I_n \\ \vdots \\ p_k I_n \\ (R_1(\lambda) \otimes I_{k_1}) \widetilde Z_1^* \\ \vdots %\\ (R_s(\lambda) \otimes I_{k_s}) \widetilde Z_s^* 
 \end{bmatrix} \text{ with } p = \begin{bmatrix}
 p_1 \\ \vdots \\ p_k 
\end{bmatrix} := \Pi^T f(\lambda) \in \C^k.
\]
The required block-UL decomposition now follows from a direct calculation if we can apply Lemma~\ref{lem:schur_R} to the partitioned matrix $\mathcal{\widetilde L}_R(\lambda) \mathcal{P}$. In order to be able to do this, we only have to establish that $p_1 = e_1^T \Pi^Tf(\lambda) \neq 0$ since then $X_1 = p_1 I_n$ is invertible (and we already have $Y_1 = R(\lambda)$). To this end, we use the definition of $\Pi$ to obtain
\[
(M-\lambda N) f(\lambda) = 0 \ \iff \ \begin{bmatrix} m_0 - \lambda n_0 & M_1 - \lambda N_1 \end{bmatrix} p = 0.
\]
Since $M_1 - \lambda N_1$ is invertible and $\dim \Ker (M-\lambda N) = 1$, the null-vector $p$ has to be of the form
\[
% \begin{bmatrix} p_1 \\ p_2 \\ \vdots \\ p_k\end{bmatrix}
	p = \alpha \, \begin{bmatrix} 1 \\ -(m_0 - \lambda n_0) (M_1 - \lambda N_1)^{-1} \end{bmatrix}, \qquad \alpha \in \C.
\]
Since $f(\lambda)$, and thus $p$, cannot never be identically zero, we get as requested that $\alpha = p_1 \neq 0$. 
\end{proof}    

\begin{proof}[Proof of Theorem~\ref{thm:evals_LR_low-rank}]
\eqref{item_a_thm:evals_LR_low-rank} Let $(\lambda_0,x)$ be an eigenpair of $R(\lambda)$, that is, $R(\lambda_0)x=0$.  Recalling that $f(\lambda)^T = \begin{bmatrix} f_0(\lambda) & \cdots & f_{k-1}(\lambda) \end{bmatrix}$ with $f_0(\lambda) \equiv 1$, we obtain
\[
 y = \bsym{\Psi}(\lambda_0)x = %\begin{bmatrix} f(\lambda) \otimes I_n \\ (R_1(\lambda) \otimes I_{k_1}) \widetilde Z_1^* \\ \vdots \\ (R_s(\lambda) \otimes I_{k_s}) \widetilde Z_s^* \end{bmatrix} x = 
 \begin{bmatrix} x \\ f_1(\lambda_0) x \\ \vdots \\ f_{k-1}(\lambda_0) x \\ (R_1(\lambda_0) \otimes I_{k_1}) \widetilde Z_1^*x \\ \vdots \\ (R_s(\lambda_0) \otimes I_{k_s}) \widetilde Z_s^*x \end{bmatrix}
\]
and thus also $y \neq 0$  due to $x \neq 0$. Using~\eqref{eq:trimmed_LR}, we see that $(\lambda_0,y)$ verifies the eigenpair equation 
\[
\mathcal{\widetilde L}(\lambda_0) y  =  \begin{bmatrix} R(\lambda_0) \\ \bsym{O} \end{bmatrix} x = \begin{bmatrix} 0 \\ 0 \end{bmatrix}.
\]

\eqref{item_b_thm:evals_LR_low-rank} Let $(\lambda_0,y)$ be an eigenpair of $\mathcal{\widetilde L}(\lambda)$. Thanks to~\eqref{eq:trimmed_LR}, it suffices to show that $y = \bsym{\Psi}(\lambda_0)x$ for some nonzero $x \in \C^n$ since that implies $R(\lambda_0) x = 0$. To this end, consider the eigenpair equation $\mathcal{\widetilde L}(\lambda_0) y = 0$ in partitioned form:
\[
\left[\begin{array}{c|c}
 \bsym A - \lambda_0 \bsym B & \bsym C - \lambda_0 \bsym D \\
 \bsym M - \lambda_0 \bsym N & \bsym{O}  \\
  \bsym Z^*  & \bsym E - \lambda_0 \bsym F
 \end{array}\right] \begin{bmatrix} y_{AB} \\ \hline y_{CD} \end{bmatrix} = \begin{bmatrix} 0 \\ 0 \\ 0 \end{bmatrix}.
\]
The second block-row expresses that $y_{AB}$ is a null vector of $\bsym M - \lambda_0 \bsym N = (M-\lambda_0 N) \otimes I_n$. By definition of $f(\lambda)$, we have $(M-\lambda N)f(\lambda) = 0$ from which $(\bsym M - \lambda \bsym N) (f(\lambda) \otimes I_n) = 0$. Since $f_0(\lambda) \equiv 1$, the $kn \times n$ matrix $F(\lambda) := f(\lambda) \otimes I_n$ has rank $n$. In addition, by rank nullity, $\dim \Ker (\bsym M - \lambda \bsym N) = n$ and so $F(\lambda)$ is a basis for $\Ker (\bsym M - \lambda \bsym N)$. Hence, there exists $x \in \C^n$ such that (recall $f_0(\lambda) \equiv 1$)
\begin{equation}\label{eq:y_AB}
y_{AB} = F(\lambda_0) x = \begin{bmatrix}  x \\ f_1(\lambda_0) x \\ \vdots \\ f_{k-1}(\lambda_0)  x\end{bmatrix}.
\end{equation}
The third block-row reads
\[
 (\bsym{E}-\lambda_0 \bsym{F}) y_{CD} = -\bsym{Z}^* y_{AB}.
\]
Together with the definitions of $\bsym{Z},\bsym{E},\bsym{F}$ and~\eqref{eq:y_AB}, we also obtain
\[
 \begin{bmatrix} (E_1 - \lambda_0 F_1) \otimes I_{k_1} \\ & \ddots \\  & & (E_s - \lambda_0 F_s) \otimes I_{k_s} \end{bmatrix} y_{CD} =
	 \begin{bmatrix} (b_1 \otimes I_{k_1})\widetilde{Z}_1^* \\ \vdots \\ (b_s \otimes I_{k_s})  \widetilde{Z}_s^*\end{bmatrix} (e_1^Tf(\lambda_0) \otimes I_n) x.
\]
Hence, isolating for $y_{CD}$ and using $R_i(\lambda) = (E_i-\lambda F_i)^{-1}b_i$, we get
\begin{equation}\label{eq:y_CD}
y_{CD} =  \begin{bmatrix} ((E_1 - \lambda_0 F_1)^{-1} \otimes I_{k_1}) (b_1 \otimes I_{k_1})\widetilde{Z}_1^* \\ \vdots \\ ((E_s - \lambda_0 F_s)^{-1} \otimes I_{k_s})  (b_s \otimes I_{k_s})  \widetilde{Z}_s^*\end{bmatrix}  x = \begin{bmatrix} (R_1(\lambda_0) \otimes I_{k_1}) \widetilde{Z}_1^* \\ \vdots \\ (R_s(\lambda_0) \otimes I_{k_s})  \widetilde{Z}_s^*\end{bmatrix}  x.
\end{equation}
Now combining~\eqref{eq:y_AB} and~\eqref{eq:y_CD}, we have indeed shown that there exists $x \in \C^n$ such that
\[
 y = \begin{bmatrix} y_{AB} \\ \hline y_{CD} \end{bmatrix} = \bsym{\Psi}(\lambda_0) x.
\]
Observe that $\bsym{\Psi}(\lambda)$ has full column rank thanks to $F(\lambda)$ being its upper block. Hence, from $y \neq 0$ it follows that $x\neq 0$ and we haven proven \eqref{item_b_thm:evals_LR_low-rank}.

\eqref{item_c_thm:evals_LR_low-rank} The statement about the algebraic multiplicity of $\lambda_0$ follows directly from~\eqref{eq:det_L_R} since the matrices $M_1 - \lambda_0 N_1$ and $E_i - \lambda_0 F_i$ are invertible by construction of $\Pi$ and by assumption, respectively, and $\alpha \neq 0$. Next, we show the equality of the geometric multiplicity of $\lambda_0$, that is, $\dim \ker R(\lambda_0)= \dim \ker  \mathcal{\widetilde L}_R(\lambda_0)$. Let $\{x_1,\hdots,x_t\}$ be a basis for $\ker R(\lambda_0)$. Then by~\eqref{item_a_thm:evals_LR_low-rank}, $\bsym{\Psi}(\lambda_0)x_i \in \ker  \mathcal{\widetilde L}_R(\lambda_0)$ for $i=1,\ldots, t$. As argued above $\bsym{\Psi}(\lambda_0)$ has full column rank, hence the vectors $\bsym{\Psi}(\lambda_0)x_1,\ldots,\bsym{\Psi}(\lambda_0)x_t$ are linearly independent and so $\dim \ker R(\lambda_0) \leq \dim \ker  \mathcal{\widetilde L}_R(\lambda_0)$. Similarly, if $\{y_1,\hdots,y_t\}$ is a basis for $\ker  \mathcal{\widetilde L}_R(\lambda_0)$, then by~\eqref{item_b_thm:evals_LR_low-rank}, $y_i = \bsym{\Psi}(\lambda_0) x_i$ for some $x_i \in \C^n$. Again due to $\bsym{\Psi}(\lambda_0)$ having full column rank, the $x_1, \ldots, x_t$ are linearly independent. Hence, $\dim \ker R(\lambda_0) \geq \dim \ker  \mathcal{\widetilde L}_R(\lambda_0)$, as we wanted to show.
\end{proof}    

\end{document}